\documentclass[10pt]{amsart}
\usepackage{amssymb}
\usepackage{amsmath,amssymb,amsfonts,amsthm,
latexsym, amscd, epsfig, epic,color}
\usepackage{mathrsfs}
\usepackage{eucal}%    caligraphic-euler fonts: \mathcal{ }
\usepackage{eufrak}%   frak-euler        fonts: \mathfrak{ }
\usepackage{xspace}
\usepackage{a4wide}
\usepackage{colordvi}

\newcounter{mycount}
\theoremstyle{plain}
\newtheorem{theorem}[mycount]{Theorem}

\newtheorem{lemma}[mycount]{Lemma}

\newtheorem{remark}{Remark}
\theoremstyle{definition}
\newtheorem{definition}{Definition}

\theoremstyle{example}

\theoremstyle{remark}
\numberwithin{equation}{section} \numberwithin{figure}{section}
\begin{document}
\title{Graph limits of random graphs from a subset of connected  $k$-trees}

\author{Michael Drmota, Emma Yu Jin$^{^*}$ and Benedikt Stufler}
\thanks{$^{^*}$Corresponding author email: yu.jin@tuwien.ac.at; Tel.: $+43(1)58801-104583$. The first author is partially supported by the Austrian Science Fund FWF, Project SFB F50-02. The second author was supported by the German Research Foundation DFG, JI 207/1-1, and is supported by the Austrian Research Fund FWF, Project SFB F50-03. The third author is supported by the German Research Foundation DFG, STU 679/1-1}
\address{Institut f\"{u}r Diskrete Mathematik und Geometrie, Technische Universit\"{a}t Wien, Wiedner Hauptstr. 8–10, 1040 Vienna, Austria}
\email{michael.drmota@tuwien.ac.at}
\address{Institut f\"{u}r Diskrete Mathematik und Geometrie, Technische Universit\"{a}t Wien, Wiedner Hauptstr. 8–10, 1040 Vienna, Austria}
\email{yu.jin@tuwien.ac.at}
\address{Unit\'e de Math\'ematiques Pures et Appliqu\'ees, \'Ecole Normale Supérieure de Lyon, 46 all\'ee d'Italie, 69364 Lyon Cedex 07, France}
\email{benedikt.stufler@ens-lyon.fr}
\maketitle
\begin{abstract}
For any set $\Omega$ of non-negative integers such that $\{0,1\}\subseteq \Omega$ and $\{0,1\}\ne \Omega$,
%which contains $0,1$ and at least one integer greater than $1$, 
we consider a random $\Omega$-$k$-tree ${\sf G}_{n,k}$ that is uniformly selected from all connected $k$-trees of $(n+k)$ vertices where the number of $(k+1)$-cliques that contain any fixed $k$-clique belongs to $\Omega$.
%We establish the scaling limit and a local weak limit of this random $\Omega$-$k$-tree ${\sf G}_{n,k}$. %Since $1$-trees are just trees, it is well-known that the random $1$-tree with $n$ vertices admits the Continuum Random Tree $\CMcal{T}_{{\sf e}}$ as the scaling limit and converges locally toward a modified Galton-Watson tree. 
We prove that %the random $\Omega$-$k$-tree 
${\sf G}_{n,k}$, scaled by $(kH_{k}\sigma_{\Omega})/(2\sqrt{n})$ where $H_{k}$ is the $k$-th Harmonic number and $\sigma_{\Omega}>0$,
%is a positive constant, 
converges to the Continuum Random Tree $\CMcal{T}_{{\sf e}}$.
%too. In particular this shows that the diameter as well as the expected distance of two vertices in a random $\Omega$-$k$-tree ${\sf G}_{n,k}$ are of order $\sqrt n$. 
Furthermore, we prove the local convergence of the rooted random $\Omega$-$k$-tree ${\sf G}_{n,k}^{\circ}$ to an infinite but locally finite random $\Omega$-$k$-tree ${\sf G}_{\infty,k}$.
\end{abstract}
{\bf \small Keywords}: {\small partial $k$-trees, Continuum Random Tree, modified Galton-Watson tree}
\section{Introduction and main results}\label{S:intro}
A {\em $k$-tree} is a generalization of a tree and can
be defined recursively: a {\it $k$-tree} is either a complete graph on $k$ vertices (= a $k$-clique) or a graph obtained from a smaller $k$-tree by adjoining a
new vertex together with $k$ edges connecting it to a $k$-clique of
the smaller $k$-tree (and thus forming a $(k+1)$-clique).
In particular, a $1$-tree is a usual tree. (Note that the parameter $k$ is always fixed.) Subgraphs of $k$-trees are called {\em partial $k$-trees}; see Figure~\ref{F:11}.

\begin{figure}[htbp]
\begin{center}
\includegraphics[scale=0.6]{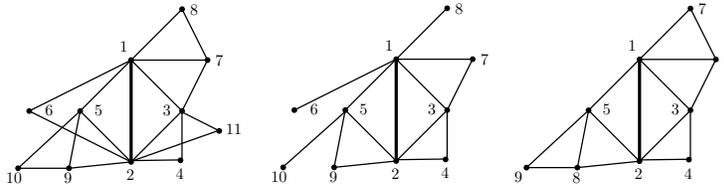}
\caption{a $2$-tree (left), a partial $2$-tree (middle) and an $\Omega$-$2$-tree (right) where $\Omega=\{0,1,2\}$.
\label{F:11}}
\end{center}
\end{figure}
A partial $k$-tree is an interesting graph from an algorithmic point of view since many
NP-hard problems on graphs have polynomial, in fact usually linear, dynamic
programming algorithms when restricted to partial $k$-trees for fixed values of $k$ \cite{AP:89,TeP:05,GK:08}; such NP-hard problems include maximum independent
set size, minimal dominating set size, chromatic number, Hamiltonian circuit,
network reliability and minimum vertex removal forbidden subgraph \cite{A:85,BB:72}.
Several graphs which are important in practice \cite{K:03}, have been shown to be
partial $k$-trees, among them are
\begin{enumerate}
\item Trees/ Forests (partial $1$-trees)
\item Series parallel networks (partial $2$-trees)
\item Outplanar graphs (partial $2$-trees)
\item Halin graphs (partial $3$-trees); see \cite{HandGT}.
\end{enumerate}
However, other interesting graph classes like planar graphs or bipartite graphs are not partial $k$-trees. On the other hand, partial $k$-trees are very interesting from a combinatorial point of view, although the enumeration of partial $k$-trees for general $k$ is still missing. The number of $k$-trees, which are ``saturated'' partial $k$-trees, has been counted in various ways; see \cite{BP:69,Moon,Foata,DS:09,HP:68,HP:73,Fowler:02,Gai,GeGai}.
As usual a graph on $n$ vertices is called {\em labelled} if the integers from $\{1,2,\ldots,n\}$ have been assigned to its vertices (one-to-one) and two labelled graphs are considered to be different if the corresponding edge sets are different.

In this paper, we introduce a subset of connected labelled  $k$-trees, called {\em $\Omega$-$k$-trees} as a first attempt to approach the profile of connected labelled partial $k$-trees by using the enumeration of labelled $k$-trees. In what follows, without specifying otherwise, we assume that $\Omega$-$k$-trees are all labelled and a random $\Omega$-$k$-tree is uniformly selected from the class of labelled $\Omega$-$k$-trees with $(n+k)$ vertices.
\begin{definition}[$\Omega$-$k$-tree]
For any set $\Omega$ of non-negative integers which contains $0,1$ and at least one integer greater than $1$, an $\Omega$-$k$-tree is a connected  $k$-tree satisfying
that the number of $(k+1)$-cliques that contain any fixed $k$-clique belongs to the set $\Omega$.
\end{definition}
A rooted $\Omega$-$k$-tree is an $\Omega$-$k$-tree rooted at a $k$-clique. If $\Omega=\mathbb{N}_0=\{0,1,2,\ldots\}$, an $\mathbb{N}_0$- $k$-tree is a $k$-tree. See Figure~\ref{F:11} for an example of $\Omega$-$2$-tree. We remark that it is necessary to allow $0\in\Omega$ since by the construction of $k$-trees, a $k$-clique is the smallest $k$-tree. We also need the condition $1\in\Omega$ because otherwise any $k$-tree, other than a single $k$-clique, is infinite, and we ignore the case $\Omega=\{0,1\}$ so that the $k$-trees are not trivial.

Darrasse and Soria \cite{DS:09} showed a Rayleigh limiting distribution for the expected distance between pairs of vertices in a random $k$-tree, as it is known for usual trees and, thus, for $1$-trees. Inspired by this results, we expect that a random $\Omega$-$k$-tree with $(n+k)$ vertices, after scaling the distances to the root by $1/\sqrt{n}$, converges to the {\em Continuum Random Tree} multiplied by a deterministic scaling factor. For $k=1$ and $\Omega=\mathbb{N}_0=\{0,1,2,\ldots\}$, this is true by a result of Aldous. Actually Aldous has proved in a series of seminal papers \cite{Ald:1,Ald:2,Ald:3} that a {\em critical Galton-Watson tree} conditioned on its size has the Continuum Random Tree (CRT) as its limiting object -- and random $1$-trees are a special case (with a Poisson offspring distribution), if the variance of the progeny is finite. The concept {\em Continuum Random Tree} was also introduced by Aldous \cite{Ald:1,Ald:2,Ald:3} and further developed by Duquesne and Le Gall \cite{Dus:1,Dus:2,Dus:3}.

Since Aldous's pioneering work on the Galton-Watson trees, the CRT has been established as the limiting object of a large variety of combinatorial structures \cite{Haas-Miermont,Stu:14,
P:14,P:142,Car:14,JS:15,JB:15,CHK:2014,Markert-Miermont,Broutin-Miermont}. A key idea in the study of these combinatorial objects is to relate them to trees endowed with additional structures by using an appropriate bijection. In the present case of $\Omega$-$k$-trees, we encode them as so-called $(k,\Omega)$-front coding trees via a bijection due to Darrasse and Soria in \cite{DS:09}, which was originally used to enumerate $k$-trees and to recursively count the distance between any two vertices in a random $k$-tree. Furthermore, in order to build a connection between the distance of two vertices in a random $\Omega$-$k$-tree and the distance of two vertices in a critical Galton-Watson tree, we need to introduce the concept of a {\em size-biased enriched tree}. This is adapted from the {\em size-biased Galton-Watson tree} which was defined by Kesten \cite{Kesten}, used by Lyons, Pemantle and Peres in \cite{Lyons}, by Addario-Berry, Devroye and Janson in \cite{J:12}, and was further generalized to the {\em size-biased $\CMcal{R}$-enriched trees} by Panagiotou, Stufler and Weller in \cite{P:142}. Our enriched tree is slightly different to the size-biased $\CMcal{R}$-enriched tree and we use their ideas in \cite{Stu:14,Stu:142} where an important step is to relate the distance between two vertices in a random graph to the distance between two blocks in a random size-biased $\CMcal{R}$-enriched tree.

When we analyze $\Omega$-$k$-trees, it turns out that it is convenient to consider the number of {\it hedra} instead of the number of vertices as the size of an $\Omega$-$k$-tree; we adopt the notions from \cite{GeGai}. A {\it hedron} is a $(k+1)$-clique in an $\Omega$-$k$-tree, and by definition an $\Omega$-$k$-tree with $n$ hedra has $(n+k)$ vertices. A {\it front} of a $k$-tree is a $k$-clique.

Our first main result establishes the weak convergence of a random $k$-tree to the CRT with respect to the Gromov-Hausdorff distance.
\begin{theorem}\label{T:crtk}
Let $\CMcal{G}_{n,k}$ be the class of labelled $\Omega$-$k$-trees with $n$ hedra and denote by ${\sf G}_{n,k}$ a random $\Omega$-$k$-tree that is uniformly selected from the class $\CMcal{G}_{n,k}$. %Let $v_n$ be a uniformly at random drawn vertex of ${\sf G}_{n,k}$.  %and by ${\sf G}_{n,k}^{\circ}$ a random $\Omega$-$k$-tree that is rooted at a uniformly chosen front.
Then
\begin{eqnarray*}
%\frac{kH_k\sigma_{\Omega}}{2\sqrt{n}}{\sf G}_{n,k}^{\circ}
%\xrightarrow{d}\CMcal{T}_{e}\quad \,\mbox{ and }\quad
({\sf G}_{n,k}, \frac{kH_k\sigma_{\Omega}}{2\sqrt{n}} d_{{\sf G}_{n,k}}) \xrightarrow{d}(\CMcal{T}_{e}, d_{\CMcal{T}_{e}})
\end{eqnarray*}
holds with respect to the Gromov-Hausdorff metric. Here $H_{k} = 1 + 1/2 + \ldots + 1/k$ denotes the $k$-th Harmonic number and $\sigma_{\Omega}$ is a positive constant. If $\Omega=\mathbb{N}_0$, the constant $\sigma_{\mathbb{N}_0}$ equals $1$.
\end{theorem}
In particular this shows that the diameter as well as the expected distance of two vertices in a random $\Omega$-$k$-tree ${\sf G}_{n,k}$ are of order $\sqrt n$ and they have up to a constant scaling factor the same limiting distribution as random $1$-trees. The constant $kH_k$ has also a natural explanation. In the proof of Theorem~\ref{T:crtk} we will partition an $\Omega$-$k$-tree into {\it rooted blocks} that constitute subsets of the same distance to the root of the $\Omega$-$k$-tree, and $kH_k$ is actually the expected length of the path from the selected {\it good node} in a block to the root of this block. Instead of the class $\CMcal{G}_{n,k}$ we could equivalently also consider the class of $\Omega$-$k$-trees with $n$ hedra that are rooted at a fixed labelled front. In Subsection~\ref{ss:red} below we will argue that the two models are equivalent and hence our results apply to both.

We recall that (partial) $1$-trees are just trees and partial $2$-trees are series-parallel graphs.
In both cases it is known \cite{Ald:3,P:14} that the CRT appears as the scaling limit
(if we scale by $c/\sqrt n$ for some positive constant $c$). We conjecture that the CRT also arises as the scaling limit of partial $k$-trees for larger $k$.

\medskip\noindent
{\bf Conjecture 1.} {\it Let $\CMcal{PT}_{n,k}$ be the class of
all connected labelled partial $k$-trees and let ${\sf PT}_{n,k}$ be a uniformly chosen random graph from $\CMcal{PT}_{n,k}$. Then ${\sf PT}_{n,k}$ converges toward the CRT in the Gromov-Hausdorff sense for every $k \ge 1$, after rescaling the metric by a factor $c_k/\sqrt{n}$ for some constant $c_k>0$.
%Then there exists a constant $c_k> 0$ such that
%\begin{eqnarray*}
%({\sf PT}_{n,k}, \frac{c_k}{\sqrt{n}}d_{{\sf PT}_{n,k}})
%\xrightarrow{d}(\CMcal{T}_{e}, d_{\CMcal{T}_{e}}).
%\end{eqnarray*}
}
\medskip

At the moment this property seems to be out of reach since there is no precise asymptotic analysis
of partial $k$-trees if $k\ge 3$. Nevertheless Theorem~\ref{T:crtk} is a strong indication that such
a property should hold. For example, if we delete $o(\sqrt n)$ edges from a random $\Omega$-$k$-tree
we (usually) do not destroy the connectivity and also the distance function might be slightly affected
but not more than $o(\sqrt n)$. Thus, if we construct partial $k$-trees in that way we still observe
a scaling limit of the above form.

\medskip

Theorem~\ref{T:crtk} describes the asymptotic global metric properties of random $k$-trees, but gives little information about asymptotic local properties. Hence we provide a second limit theorem that establishes the local weak convergence of the random $\Omega$-$k$-tree ${\sf G}_{n,k}$ toward an infinite but locally finite $\Omega$-$k$ tree ${\sf G}_{\infty,k}$. This type of convergence describes the asymptotic behaviour of neighborhoods around a randomly chosen front.

%It is named after the authors of the fundamental work \cite{BeSc} on the recurrence of distributional limits of random planar graphs.
% but contains no information about local properties such as the limiting distribution of the degree of a randomly drawn vertex and we will address this problem in Theorem~\ref{T:local}.
%We denote by $(G,v)$ the graph $G$ rooted at a vertex $v$. Our second main result establishes the local weak convergence of a random rooted $\Omega$-$k$-tree to an infinite random rooted $\Omega$-$k$-tree.
\begin{theorem}
	\label{T:local}
	Let $\CMcal{G}_{n,k}$ be the class of labelled $\Omega$-$k$-trees with $n$ hedra and denote by ${\sf G}_{n,k}^\circ$ a random $\Omega$-$k$-tree that is uniformly selected from the class $\CMcal{G}_{n,k}$ and then rooted at a uniformly at random chosen front. Then, as $n$ tends to infinity, the random graph ${\sf G}_{n,k}^\circ$ converges in the local-weak sense toward a front-rooted infinite $\Omega$-$k$-tree ${\sf G}_{\infty,k}$, that is,
\[
G_{n,k}^\circ \xrightarrow{d} {\sf G}_{\infty,k}.
\]

%	Then, as $n$ tends to infinity, the random graph ${\sf G}_{n,k}$ converges in the Benjamini-Schramm sense toward a pointed limit $\Omega$-$k$-tree $({\sf G}_{\infty,k}, v_\infty)$.	
	%Then, as $n$ tends to infinity, the random graph ${\sf G}_{n,k}$ converges in the Benjamini-Schramm sense toward a pointed limit $\Omega$-$k$-tree $({\sf G}_{\infty,k}, v_\infty)$.
\end{theorem}

Our proof of Theorem~\ref{T:local} builds on the classical local convergence of simply generated trees toward a modified Galton--Watson tree. See for example Theorem 7.1 in Janson's survey \cite{Janson:11}, which unifies some results by Kennedy \cite{Kennedy}, Aldous and Pitman \cite{AldPit}, Grimmett \cite{Grimmett}, Kolchin \cite{Kolchin}, Kesten \cite{Kesten}, Aldous \cite{Ald:2}, Jonsson and Stef\'{a}nsson \cite{JS:11} and Janson, Jonsson and Stef\'{a}nsson \cite{JJS:11}.% for different cases of Galton-Watson trees.

A result similar to Theorem~\ref{T:local} is known for partial $2$-trees since series-parallel graphs
belong to the family of subcritical graph classes \cite{Stu:142,GeWag}. Therefore we can also formulate the following conjecture.

\medskip\noindent
{\bf Conjecture 2.} {\it The random labelled partial $k$-tree ${\sf PT}_{n,k}$ converges in the local-weak sense for every $k \ge 1$. That is, the neighborhoods of a random front in ${\sf PT}_{n,k}$ converge weakly toward the neighborhoods of a front-rooted infinite partial $k$-tree ${\sf PT}_{\infty,k}$ as $n \to \infty$. %there exists an infinite
%random rooted partial $k$-tree ${\sf PT}_{\infty,k}$ such that
%	\[
%({\sf PT}_{n,k}, v_n) \xrightarrow{d} ({\sf PT}_{\infty,k}, v_{\infty})
%	\]
%as $n\to \infty$.
}

\medskip

The plan of the paper is as follows. In Section~\ref{S:comb} we recall the combinatorial background for $\Omega$-$k$-trees, introduce the Boltzmann sampler -- a method of generating efficiently a uniform random combinatorial object, describe Darrasse and Soria's algorithm on computing the distances between two vertices in an $\Omega$-$k$-tree, present Aldous's result on the convergence of critical Galton-Watson trees to the CRT $\CMcal{T}_{{\sf e}}$, and recall the notion of local convergence. In Section~\ref{S:proof} we prove our first main result -- Theorem~\ref{T:crtk}, and in Section~\ref{L:proof} our second main result -- Theorem~\ref{T:local}.

\section{Combinatorics, Boltzmann Samplers and Graph Limits}\label{S:comb}
Let $\Omega\subset \mathbb{N}_0$ denote a set of non-negative integers which contains $0,1$ and at least one integer greater than $1$. We will review the generating function approach from \cite{DS:09} to count the number ${\sf Par}_{k,\Omega}(n)$ of $\Omega$-$k$-trees. The key ingredient to count the number ${\sf Par}_{k,\Omega}(n)$ is a bijection between rooted $\Omega$-$k$-trees and {\em $(k,\Omega)$-front coding trees}; see \cite{DS:09}.
\begin{definition}[$(k,\Omega)$-front coding tree]\label{D:kcoding}
For any set $\Omega$ of non-negative integers which contains $0,1$ and at least one integer greater than $1$, a {\em $(k,\Omega)$-front coding tree} of size $n$ is a tree $T$ consisting of $(kn+1)$ white nodes and $n$ black nodes which satisfies:
\begin{enumerate}
\item $T$ is rooted at a white node, every white node has only black nodes as children and every black node has only $k$ white nodes as children.
\item The number of black children of the white root belongs to the set $\Omega$ and the number of black children of any other white node belongs to the set $\{i\,\vert\,i+1\in\Omega,i\ge 0\}$.
\item The white root of $T$ is labeled by a $k$-subset $A$ of $[n+k]=\{1,2,\ldots,n+k\}$ and the black nodes are labeled by the integers from the set $[n+k]-A$ such that for every white node, the subtrees stemming from its black children are not ordered between themselves.
\end{enumerate}
The labels on the white root and black nodes determine the labels on the rest white nodes. We start from the white root and recursively label other white nodes. For every white node, we label it with a set $\{r_1,\ldots,r_{i-1},r,r_{i+1},\ldots,r_k\}$ if the white node is the $i$-th child (from left to right) of a black node labeled by $r$ and the white parent of this black node is labeled with the set $\{r_1,\ldots,r_k\}$.

If the white root of a $(k,\Omega)$-front coding tree has precisely one black child, we call it {\em reduced $(k,\Omega)$-front coding tree}.
\end{definition}
We first list all important notations of $\Omega$-$k$-trees and $(k,\Omega)$-front coding trees that are necessary in our argument.
\begin{enumerate}
\item $\CMcal{G}_{n,k}$: the class of labelled $\Omega$-$k$-trees with $n$ hedra.
\item ${\sf G}_{n,k}$: a random $\Omega$-$k$-tree that is uniformly selected from the class $\CMcal{G}_{n,k}$.
\item ${\sf G}_{n,k}^{\circ}$: a random $\Omega$-$k$-tree ${\sf G}_{n,k}$ that is rooted at a uniformly chosen front.
\item $\CMcal{G}_{n,k}^{\square}$: the class of labelled $\Omega$-$k$-trees with $n$ hedra that are rooted at a fixed front $\{1,2,\ldots,k\}$.
\item $\CMcal{G}_{n,k}^{\bullet}$: the class of labelled $\Omega$-$k$-trees with $n$ hedra that are rooted at a fixed front $\{1,2,\ldots,k\}$ and this root front is contained in only one hedron.
\item $\CMcal{C}_{n,k}$: the class of $(k,\Omega)$-front coding trees of size $n$ that are rooted at a white node $\{1,2,\ldots,k\}$.
\item ${\sf C}_{n,k}$: a random $(k,\Omega)$-front coding tree that is uniformly selected from $\CMcal{C}_{n,k}$.
\item $\CMcal{B}_{n,k}$: the class of reduced $(k,\Omega)$-front coding trees of size $n$ that are rooted at a white node $\{1,2,\ldots,k\}$.
\item ${\sf B}_{n,k}$: a random reduced $(k,\Omega)$-front coding tree that is uniformly selected from $\CMcal{B}_{n,k}$.
\item ${\sf G}_{n,k}^{\bullet}$: a random $\Omega$-$k$-tree that uniquely corresponds to ${\sf B}_{n,k}$ under the bijection $\varphi$ where the bijection $\varphi$ will be shown in subsection~\ref{ss:enum}. This is equivalent to uniformly choose a random $\Omega$-$k$-tree from the class $\CMcal{G}_{n,k}^{\bullet}$.
\item ${\sf G}_{n,k}^{\square}$: a random $\Omega$-$k$-tree that uniquely corresponds to ${\sf C}_{n,k}$ under the bijection $\varphi$. This is equivalent to uniformly choose a random $\Omega$-$k$-tree from the class $\CMcal{G}_{n,k}^{\square}$.
\end{enumerate}
\subsection{A one-to-one correspondence $\varphi$}\label{ss:enum}
We recall that a rooted $\Omega$-$k$-tree is an $\Omega$-$k$-tree rooted at a front (or equivalently a $k$-clique). For the case $\Omega=\mathbb{N}_0$, we simply call a $(k,\mathbb{N}_0)$-front coding tree a {\em $k$-front coding tree}. By Definition~\ref{D:kcoding}, a $k$-front coding tree is a bipartite tree of black and white nodes which is rooted at a white node and where every black node has precisely $k$ successors. We will present a one-to-one correspondence $$\varphi:\CMcal{G}_{n,k}^{\square}\rightarrow \CMcal{C}_{n,k}$$
when $\Omega=\mathbb{N}_0$, that is, a one-to-one correspondence $\varphi$ between rooted $k$-trees and $k$-front coding trees. The bijection $\varphi$ holds for any $\Omega$-$k$-tree when we specify the outdegrees of the white nodes in the corresponding $(k,\Omega)$-front coding tree.

The correspondence $\varphi$ will be built in a way that black nodes in a $k$-front coding tree correspond to hedra in a $k$-tree. Every black node also gets a label which is equal to the label of one of the vertices of the corresponding hedron. A white node in a $k$-front coding tree corresponds to a front of the $k$-trees and is labelled by the set $\{a_1,a_2,\ldots,a_k\}$ of labels of the corresponding front. A black node connects with a white node if the corresponding hedron contains the corresponding front and the label of the black node is just the label of the vertex that is not contained in the front. Thus, if we start with the root front of the $k$-tree we can recursively build up a corresponding $k$-front coding tree; see Figure~\ref{F:1}.

\begin{figure}[htbp]
\begin{center}
\includegraphics[scale=0.7]{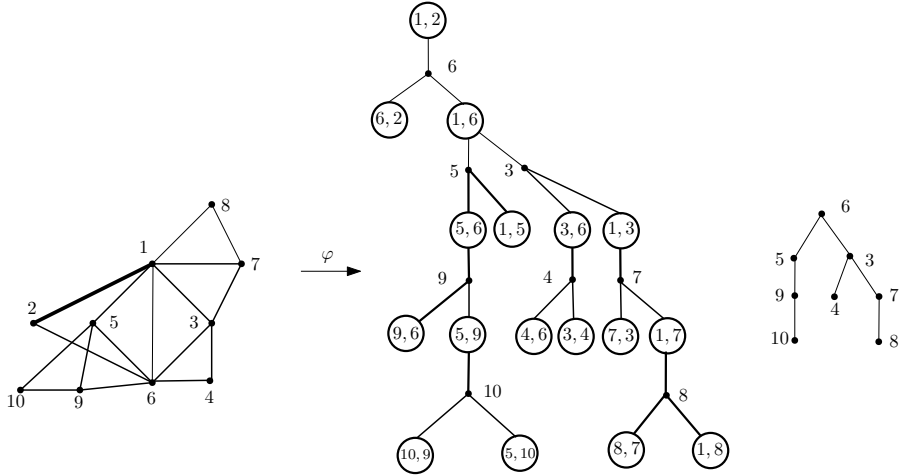}
\caption{When $\Omega=\{0,1,2,3\}$, an $\Omega$-$2$-tree rooted at a front whose vertices are labelled by $1,2$ (left) and the corresponding $(2,\Omega)$-front coding
tree ${\sf C}_{n,2}$ rooted at a white node labelled by $\{1,2\}$ (middle); finally the derived black tree ${\sf T}_{n}$ consists just of only black nodes of ${\sf C}_{n,2}$ (right).
\label{F:1}}
\end{center}
\end{figure}

With the help of this correspondence $\varphi$, the problem of counting the $\Omega$-$k$-trees with $n$ hedra is reduced to count the corresponding $(k,\Omega)$-front coding trees with $n$ black nodes. We use the notation {\it $\circ$-rooted $(k,\Omega)$-front coding trees}
if the white root node has a fixed label and use the notation {\em $\circ-\bullet$ $(k,\Omega)$-front coding tree} if the white root of a reduced $(k,\Omega)$-front coding tree has a fixed label.

Let $\CMcal{G}_k^{\square}$ be the class of $\Omega$-$k$-trees rooted at a fixed front $\{a_1,a_2,\ldots,a_k\}$, furthermore let $\CMcal{C}_k$ be the class of the $(k,\Omega)$-front coding trees and $\CMcal{B}_k$ be the class of $\circ-\bullet$ $(k,\Omega)$-front coding trees. In fact, the correspondence $\varphi$ also establishes the relation $\varphi:\CMcal{G}_k^{\square}\rightarrow\CMcal{C}_k$. Furthermore, every $(k,\Omega)$-front coding tree can be identified as a set of $\circ-\bullet$ $(k,\Omega)$-front coding trees with the outdegree set $\Omega$, which leads to the relation
\begin{align}\label{E:spe1}
\CMcal{C}_k=\textsc{Set}_{\Omega}(\CMcal{B}_k).
\end{align}
In terms of exponential generating functions (where the size is always the number of black nodes), we thus get
\begin{align}\label{E:gespe1}
C_k(x)=\sum_{i\in \Omega}\frac{(B_k(x))^i}{i\,!}.
\end{align}
We recall that ${\sf C}_{n,k}$ is a random $(k,\Omega)$-front coding tree that is uniformly selected from the $(k,\Omega)$-front coding trees of size $n$. We denote by ${\sf L}_{n,k}$ one of the largest $\circ-\bullet$ $(k,\Omega)$-front coding trees that is contained in ${\sf C}_{n,k}$ and denote by $L_{n,k}$ the size of ${\sf L}_{n,k}$. By employing a unified analytic framework given by Xavier Gourdon \cite{G:98}, from (\ref{E:spe1}) we can prove that for any sufficiently small $\varepsilon$ such that $\varepsilon>0$, one has
\begin{align}\label{E:largest}
\mathbb{P}[n-L_{n,k}\ge n^{\varepsilon}]\rightarrow 0.
\end{align}
Now we continue to decompose the $\circ-\bullet$ $(k,\Omega)$-front coding tree. Let $\CMcal{C}_k^{\circ}$ be the class of $\circ$-rooted $(k,\Omega)$-front coding trees that are contained in the $\circ-\bullet$ $(k,\Omega)$-front coding trees. Then every tree from $\CMcal{C}_k^{\circ}$ can be identified as a set of $\circ-\bullet$ $(k,\Omega)$-front coding trees with the outdegree set $\Omega_{\scriptsize{\mbox{out}}}$ of the white nodes where $\Omega_{\scriptsize{\mbox{out}}}=\{i\,\vert\,i+1 \in \Omega,i\ge 0\}$. Moreover, every $\circ-\bullet$ $(k,\Omega)$-front coding tree can be decomposed into a $k$-tuple of $\circ$-rooted $(k,\Omega)$-front coding trees. This yields the following specification:
\begin{eqnarray}\label{E:symlab}
\CMcal{B}_k=\{\bullet\}*\textsc{Seq}_k(\CMcal{C}_k^{\circ}) \quad\,\,\mbox{and}\,\quad\CMcal{C}_k^{\circ}
=\textsc{Set}_{\Omega_{\scriptsize{\mbox{out}}}}
(\CMcal{B}_k).
\end{eqnarray}
In terms of exponential generating functions, one gets
\begin{eqnarray}\label{E:symfun}
B_k(x)=x\cdot C_k^{\circ}(x)^k\quad\,\,\mbox{and}\,\quad C_k^{\circ}(x)=\sum_{\substack{i+1\in\Omega\\i\ge 0}}\frac{(B_k(x))^i}{i\,!}.
\end{eqnarray}
In particular $B_k(x)$ satisfies
\begin{equation}\label{E:Bkn2}
B_k(x)=x(\sum_{\substack{i+1\in\Omega\\i\ge 0}}\frac{(B_k(x))^i}{i\,!})^k.	
\end{equation}
Consequently there exists a unique positive dominant singularity $\rho_{k,\Omega}$ of $B_k(x)$ such that
\begin{align}\label{E:domieq}
\sum_{\substack{i+1\in\Omega\\ i\ge 1}}\frac{(ki-1)}{i\,!}
(B_k(\rho_{k,\Omega}))^i=1\,\quad\mbox{ and }\quad\, B_k(\rho_{k,\Omega})<\infty.
\end{align}
It follows immediately from (\ref{E:gespe1}) and (\ref{E:symfun}) that $C_k(\rho_{k,\Omega})<\infty$ and $C_k^{\circ}(\rho_{k,\Omega})<\infty$. We set $b_{k,\Omega}(n)=n![x^n]B_k(x)$ and $c_{k,\Omega}(n)=n![x^n]C_k(x)$ which counts the number of $\circ$-rooted $(k,\Omega)$-front coding trees of $n$ black nodes and the root $\circ$ has a fixed label $\{a_1,a_2,\ldots,a_k\}$. Since there are ${n+k \choose k}$ ways to choose the root $\{a_1,a_2,\ldots,a_k\}$, the number of $\Omega$-$k$-trees having $n$ hedra that are rooted at a front is
\begin{eqnarray}\label{E:count}
(kn+1){\sf Par}_{k,\Omega}(n)={n+k\choose k}c_{k,\Omega}(n),
\end{eqnarray}
and ${\sf Par}_{k,\Omega}(n)$ can be derived from (\ref{E:count}) for any specific $\Omega$. One can analyze the asymptotic behaviors of $b_{k,\Omega}(n)$ and $c_{k,\Omega}(n)$ from (\ref{E:Bkn2}); see \cite{Drmotabook,FS}, which yields
\begin{align}
\label{eq:required}
b_{k,\Omega}(n)\sim d_1 n^{-3/2}n!(\rho_{k,\Omega})^{-n}\quad\mbox{ and }\quad c_{k,\Omega}(n)\sim d_2 n^{-3/2}n!(\rho_{k,\Omega})^{-n}
\end{align}
for some positive constants $d_1,d_2$. Together with (\ref{E:count}) this leads to
\begin{align*}
{\sf Par}_{k,\Omega}(n)\sim \frac{d_2\,n^{n+k-2}}{k\cdot k!}(\rho_{k,\Omega})^{-n}.
\end{align*}
Furthermore, one can also estimate the number ${\sf U}_{k,\Omega}(n)$ of unlabeled $\Omega$-$k$-trees.
\begin{align*}
{\sf U}_{k,\Omega}(n)\sim d_3 n^{-5/2} (\tau_{k,\Omega})^{-n}
\end{align*}
where $d_3$ is a positive constant and $\tau_{k,\Omega}$ is the dominant singularity of $A_k(z)$ that is given by
\begin{align}\label{E:1}
A_k(z)=z\sum_{k\in\Omega}\sum_{\lambda\vdash k}\frac{(kA_k(z))^{\lambda_1}}{\lambda_1!}
\frac{(kA_k(z^2))^{\lambda_2}}{\lambda_2!2^{\lambda_2}}\cdots
\frac{(kA_k(z^k))^{\lambda_k}}{\lambda_k!k^{\lambda_k}},
\end{align}
in which $\lambda\vdash k$ is a partition of $k$ and by $\lambda_i$ we denote the number of parts in $\lambda$ with length $i$. The dominant singularity $z=\tau_{k,\Omega}$ is the unique solution of (\ref{E:1}) and
\begin{align*}
\frac{1}{k}=z\sum_{k\in\Omega}\sum_{\substack{\lambda\vdash k\\
\lambda_1\ge 1}}\frac{(kA_k(z))^{\lambda_1-1}}{(\lambda_1-1)!}
\frac{(kA_k(z^2))^{\lambda_2}}{\lambda_2!2^{\lambda_2}}\cdots
\frac{(kA_k(z^k))^{\lambda_k}}{\lambda_k!k^{\lambda_k}}.
\end{align*}
For the case $\Omega=\mathbb{N}_0$, the number ${\sf U}_{k,\mathbb{N}_0}(n)$ of unlabeled $k$-trees is estimated in \cite{DJ:14}.
\begin{remark}
If $\Omega=\mathbb{N}_0$, it was shown in \cite{BP:69,Moon,Foata,DS:09} that the number ${\sf Par}_{k,\mathbb{N}_0}(n)$ of $\mathbb{N}_0$- $k$-trees having $n$ hedra is given by
\begin{equation}\label{E:Bkn}
{\sf Par}_{k,\mathbb{N}_0}(n)={n+k\choose k}(kn+1)^{n-2},
\end{equation}
thus, asymptotically by ${\sf Par}_{k,\mathbb{N}_0}(n)\sim {n^k}(kn)^{n-2}e^{1/k}(k!)^{-1}$ as $n\to\infty$. By applying the Lagrange inversion formula on (\ref{E:Bkn2}) for the case $\Omega=\mathbb{N}_0$, we obtain that the number of $\circ-\bullet$ $(k,\mathbb{N}_0)$-front coding trees with $n$ black nodes where the root $\circ$ has a fixed label $\{a_1,a_2,\ldots,a_k\}$, is
\begin{eqnarray}\label{E:enuB}
b_{k,\mathbb{N}_0}(n)=n!\,[x^n]B_k(x)=(n-1)![x^{n-1}]\exp(knx)=(kn)^{n-1}
\end{eqnarray}
and the number of $\circ$-rooted $(k,\mathbb{N}_0)$-front coding trees with $n$ black nodes where the root $\circ$ has a fixed label $\{a_1,a_2,\ldots,a_k\}$ is
\begin{eqnarray}\label{E:enuC}
c_{k,\mathbb{N}_0}(n)=n!\,[x^n]C_k(x)=(n-1)![x^{n-1}]\exp((kn+1)x)
=(kn+1)^{n-1}.
\end{eqnarray}
In view of (\ref{E:enuC}), the closed formula (\ref{E:Bkn}) for ${\sf Par}_{k,\mathbb{N}_0}(n)$ is proved. It follows from (\ref{E:Bkn2}) that the dominant singularity of $B_k(x)$ for the case $\Omega=\mathbb{N}_0$ is $\rho_{k,\mathbb{N}_0}=(ek)^{-1}$ and $B_k(\rho_{k,\mathbb{N}_0})=k^{-1}$; see \cite{DS:09,DJ:14} for details.
\end{remark}
\subsection{Reduction of Theorem~\ref{T:crtk}}\label{ss:red}
We reduce Theorem~\ref{T:crtk} to the scaling limit of a random rooted $\Omega$-$k$-trees where the root front has vertices labelled by $1,2,\ldots,k$.

Since any $\Omega$-$k$-tree with $n$ hedra has the same number, $ (kn+1)$, of fronts, it makes no difference whether we root ${\sf G}_{n,k}$ at a uniformly at random chosen front, or if we select an element from the class $\CMcal{G}_{n,k}^\circ$ uniformly at random. From (\ref{E:count}) and the bijection $\varphi$ we find that for all $g\in \CMcal{G}_{n,k}^{\circ}$ and $c\in \CMcal{C}_{n,k}$ we have
$$\mathbb{P}[{\sf G}_{n,k}^{\circ}=g]=\binom{n+k}{k}^{-1}\mathbb{P}[{\sf C}_{n,k}=c]
=\binom{n+k}{k}^{-1}\mathbb{P}[{\sf G}_{n,k}^{\square}=\varphi^{-1}(c)],$$
which means that the probability to uniformly choose a front-rooted $\Omega$-$k$-tree is equal to the probability to first uniformly choose a rooted $\Omega$-$k$-tree from $\CMcal{G}_{n,k}^{\square}$ and then replace the label $\{1,2,\ldots,k\}$ on the root by a uniformly chosen $k$-subset of $[n]$. Since the relabeling will not change the distance of two vertices in the graph and will not change the probability to choose an $\Omega$-$k$-tree of a given shape, without loss of generality we can fix the labeling of the root front and consider the random $\Omega$-$k$-tree ${\sf G}_{n,k}^{\square}$. That is, it suffices to prove Theorem~\ref{T:crtk} for the random $\Omega$-$k$-tree that is uniformly selected from $\CMcal{G}_{n,k}^{\square}$. This is equivalent to uniformly choose a $(k,\Omega)$-front coding tree ${\sf C}_{n,k}$ from $\CMcal{C}_{n,k}$ and consider the corresponding random $\Omega$-$k$-tree ${\sf G}_{n,k}^{\square}=\varphi^{-1}({\sf C}_{n,k})$. %Furthermore, in order to make a uniformly at random choice of a vertex of $\CMcal{G}_{n,k}$, we may equivalently mark a uniformly at random drawn vertex of the marked front of ${\sf G}_{n,k}^{\circ}$.

%\begin{lemma}
%	\label{L:simplify}
%	The random $\Omega$-$k$-tree ${\sf G}_{n,k}$ together with a uniformly at random drawn root vertex $v_n$ is distributed like ${\sf G}_{n,k}^\square$ rooted.
%\end{lemma}

We can further reduce Theorem~\ref{T:crtk} to the scaling limit of a random rooted $\Omega$-$k$-tree such that the root front is contained in only one hedron. That is, a random rooted $\Omega$-$k$-tree that uniquely corresponds to a $\circ-\bullet$ $(k,\Omega)$-coding tree from $\CMcal{B}_k$. We put this in Section~\ref{S:proof} after we introduce the Gromov-Hausdorff metric in subsection~\ref{ss:GH}.

Since $\CMcal{B}_k$ has a proper recursive specification (\ref{E:symfun}), these random objects can be constructed (or sampled) by a so-called Boltzmann sampler $\Gamma B_k(x)$.
\subsection{Boltzmann Sampler}
Boltzmann samplers provide a way to efficiently generate a combinatorial object at random. They were introduced by Duchon, Flajolet, Louchard and Schaeffer \cite{DF:04} and were further developed by Flajolet, Fusy and Pivoteau \cite{FFP:07}. Here we refer the readers to their papers \cite{DF:04,FFP:07} for a detailed description of the Boltzmann samplers.
We just mention that the Boltzmann sampler $\Gamma M(x)$ is a random generator which chooses an
object $c\in \CMcal{M}$ with probability $\mathbb{P}(\Gamma M(x)=c)=x^{\vert c\vert}/(M(x)\vert c\vert!)$, where
$M(x)$ denotes the exponential generating function of $c\in \CMcal{M}$ and the parameter $x>0$ ist such that $0<M(x)<\infty$.
An important property of Boltzmann samplers is that they generate objects conditioned on output size $n$ uniformly.

More precisely we will describe a  Boltzmann sampler $\Gamma B_k(x)$ with parameter $x=\rho_{k,\Omega}$
(which is possible since $B_k(\rho_{k,\Omega})<\infty$). We denote by $\xi_{\circ}$ the random variable with probability distribution
\begin{align}\label{E:bullet}
\mathbb{P}[\xi_{\circ}=i]&=\frac{1}{C_k^{\circ}(\rho_{k,\Omega})}\frac{(B_k(\rho_{k,\Omega}))^i}
{i\,!}
\mbox{ if }\, i\in\Omega_{\scriptsize{\mbox{out}}}\, \quad\mbox{ and }\, \quad \mathbb{P}[\xi_{\circ}=i]=0\,\mbox{ otherwise}.
\end{align}
\begin{lemma}\label{L:Boltz}
The following recursive procedure $\Gamma B_k(\rho_{k,\Omega})$ terminates almost surely and draws a random $\circ-\bullet$ $(k,\Omega)$-front coding tree according to the Boltzmann distribution with parameter $\rho_{k,\Omega}$, i.e., any $\circ-\bullet$ $(k,\Omega)$-front coding tree of size $n$ is drawn with probability $\rho_{k,\Omega}^n/(n!\,B_k(\rho_{k,\Omega}))$.
\end{lemma}
\begin{tabbing}
\rule{0cm}{0mm}$\Gamma B_k(\rho_{k, \Omega})$: $x_1\leftarrow$ a black node $\bullet$\\
\rule{1.5cm}{0mm}for $i:=1\rightarrow k$\\
\rule{2.0cm}{0mm}$x_2\leftarrow$ a single white node $\circ$\\
\rule{2.0cm}{0mm}merge $x_2$ into $x_1$ by adding an edge $\bullet-\circ$\\
\rule{2.0cm}{0mm}$m\leftarrow \xi_{\circ}$ and $m\in\Omega_{\scriptsize{\mbox{out}}}$ \\
\rule{2.0cm}{0mm}\=$\CMcal{F}\leftarrow$ an $m$-tuple $(\Gamma B_k(\rho_{k,\Omega}),\ldots,\Gamma B_k(\rho_{k,\Omega}))$, \\
\rule{2.0cm}{0mm}\=drop the labels\\
\rule{2.0cm}{0mm}\=merge $\CMcal{F}$ into $x_1$ by connecting $x_2$ to the roots of $\CMcal{F}$\\
\rule{1.5cm}{0mm}$x_1\leftarrow$ label the black nodes of $x_1$ uniformly at random\\
\rule{1.5cm}{0mm}return $x_1$
\end{tabbing}
\begin{remark}
	\label{re:boaut}
Boltzmann sampler can be compiled automatically from combinatorial specifications. In the present case of $\Omega$-$k$-trees, the specification given in (\ref{E:symlab}) involves product $*$ and $\textsc{Set}_{\Omega_{\scriptsize{\mbox{out}}}}$, consequently we need the rules of $\textsc{Seq}_k$ and $\textsc{Set}_{\Omega_{\scriptsize{\mbox{out}}}}$ for the inductive construction of Boltzmann sampler $\Gamma F(x)$ and $\Gamma C_k^{\circ}(x)$, which are
\begin{center}
  \begin{tabular}{c l}
  \hline\hline
    \mbox{Construction} & \mbox{Generator}\\
    \hline
    $\CMcal{F}=\textsc{Seq}_k(\CMcal{C}_k^{\circ})$ & return the $k$-tuple $(\Gamma C_k^{\circ}(x),\cdots,\Gamma C_k^{\circ}(x))$  relabeled uniformly at random.\\
    $\CMcal{C}_k^{\circ}=\textsc{Set}_{\Omega_{\scriptsize{\mbox{out}}}}(\CMcal{B}_k)$ & $m\leftarrow\xi_{\circ}$ and $m\in \Omega_{\scriptsize{\mbox{out}}}$, return the $m$-tuple $(\Gamma B_k(x),\ldots,\Gamma B_k(x))$ \\ & relabeled uniformly at random.\\
    \hline\hline
  \end{tabular}
\end{center}
For the case $\Omega=\mathbb{N}_0$, we have $\CMcal{C}_k=\CMcal{C}^{\circ}_k=\textsc{Set}(\CMcal{B}_k)$ and from (\ref{E:bullet}) it follows that $\xi_{\circ}$ is Poisson distributed with parameter $B_k(\rho_{k,\mathbb{N}_0})=k^{-1}$ where $\rho_{k,\mathbb{N}_0}=(ek)^{-1}$.
\end{remark}
Note that $(k,\Omega)$-front coding trees satisfy the specification (\ref{E:symfun}), but they do not represent the distance relation in the $\Omega$-$k$-trees; see Figure~\ref{F:1}. Since we have fixed the label on the white root $\circ$, which is $\{1,2,\ldots,k\}$, the labels on the black nodes of $\Gamma B_k(\rho_{k,\Omega})$ determine the corresponding labels on the other white nodes.
%%%%%%%%%%%%%%
\subsection{$\Omega$-$k$-tree distance algorithm}\label{ss:kdis}
For a random $(k,\Omega)$-front coding tree ${\sf C}_{n,k}$, ${\sf G}_{n,k}^{\square}$ is
the corresponding $\Omega$-$k$-tree under the bijection $\varphi^{-1}:\CMcal{C}_{n,k}\rightarrow \CMcal{G}_{n,k}^{\square}$ in subsection~\ref{ss:enum}. So ${\sf G}_{n,k}^{\square}$ is rooted at the front $\{1,2,\ldots,k\}$.

We use the notation $(i^m,j^{k-m})$ to represent the sequence of length $k$ that has $m$ occurrences of $i$ and $(k-m)$ occurrences of $j$.
Here we shall consider the distances to the vertex $1$ in an $\Omega$-$k$-tree ${\sf G}_{n,k}^{\square}$. Darrasse and Soria \cite{DS:09} provided an algorithm to calculate the distances to the vertex $1$ in an $\Omega$-$k$-tree ${\sf G}_{n,k}^{\square}$ by marking the distances on the corresponding $(k,\Omega)$-front coding tree ${\sf C}_{n,k}$, which is similar to the algorithm given by Proskurowski in \cite{Pros}. Note that every black node of the $(k,\Omega)$-front coding tree is related to a vertex of the corresponding $\Omega$-$k$-tree via its label, and the vertices that label a white node of the $(k,\Omega)$-front tree represent $k$ vertices that constitute a front of the corresponding $\Omega$-$k$-tree. We recall Darrasse and Soria's algorithm.
\begin{tabbing}
{\bf Algorithm $1$}: Distances in an $\Omega$-$k$-tree\\
\rule{1cm}{0mm}Input: a $(k,\Omega)$-front coding tree ${\sf C}$ and \\
\rule{2.1cm}{0mm}a sequence $(a_i)_{i=1}^k=(0,1^{k-1})$\\
\rule{1cm}{0mm}Output: an association table (vertex, distance)\\
\rule{1cm}{0mm}$p:=\min\{a_i\}_{i=1}^k+1$ and $A=\varnothing$\\
\rule{1cm}{0mm}for all sons $v$ of the root ${\sf C}$ do\\
\rule{1.5cm}{0mm}$A:=A\cup \{(v,p)\}$\\
\rule{1.5cm}{0mm}for $i:=1\rightarrow k$ do\\
\rule{1.5cm}{0mm}\=$A\leftarrow A\,\cup$ the recursive call on the $i$-th son \\
\rule{2.4cm}{0mm}\= of $v$ and $(a_1,\ldots,a_{i-1},p,a_{i+1},\ldots,a_k)$\\
\rule{1cm}{0mm}return $A$
\end{tabbing}
If we implement this algorithm on the $(2,\Omega)$-front coding tree (middle) in Figure~\ref{F:1}, we get a distance table marked on every black node in Figure~\ref{F:2}. The distance sequences on the white nodes help us to recursively mark the distances on the black nodes.
\begin{remark}
Based on this distance algorithm, Darrasse and Soria used the generating function approach to show a Rayleigh limiting distribution for the expected distances between pairs of vertices in a random $k$-tree; see \cite{DS:09}.
\end{remark}
\begin{figure}[htbp]
\begin{center}
\includegraphics[scale=0.7]{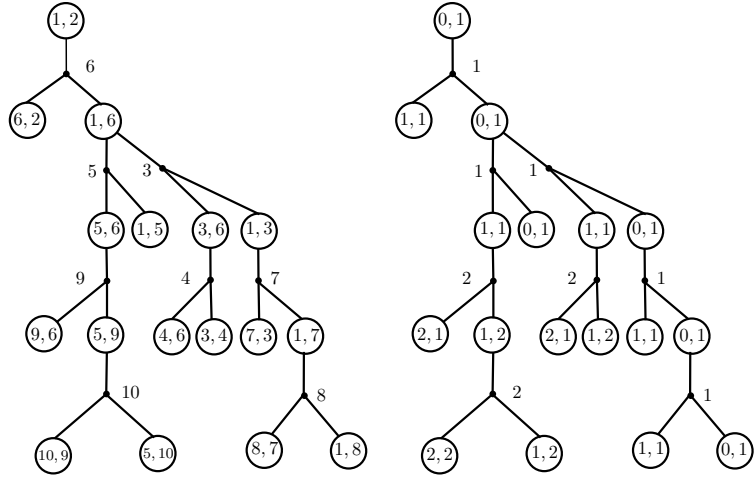}
\caption{When $\Omega=\{0,1,2,3\}$, a $(2,\Omega)$-front coding tree (left) and the corresponding distance table on every black node (right).
\label{F:2}}
\end{center}
\end{figure}
\subsection{Gromov-Hausdorff convergence and the CRT}\label{ss:GH}
Let ${\sf e}=({\sf e}_t)_{0\le t\le 1}$ denote the {\em Brownian excursion of duration one}. Then this (random) continuous function ${\sf e}$ induces a pseudo-metric on the interval $[0,1]$ by
\begin{eqnarray*}
d_{{\sf e}}(u,v)={\sf e}(u)+{\sf e}(v)-2\inf_{u\le s\le v} {\sf e}(s)
\end{eqnarray*}
for $u\le v$. This defines a metric on the quotient $\CMcal{T}_{{\sf e}}=[0,1]/\!\!\sim$ where $u\sim v$ if and only if $d_{{\sf e}}(u,v)=0$. The corresponding random pointed metric space $(\CMcal{T}_{{\sf e}}, d_{{\sf e}},r_0(\CMcal{T}_{{\sf e}}))$, where $r_0(\CMcal{T}_{{\sf e}})$ is the equivalence class of the origin, is the {\it Continuum Random Tree} (CRT). We will simply use $\CMcal{T}_{{\sf e}}$ to denote the CRT. Recall that the isometry classes of (pointed) compact metric spaces $\mathbb{K}(\mathbb{K}^{\bullet})$, where a pointed compact space is a triple $(X,d,r)$, where $(X,d)$ is a metric space and $r\in X$ is a distinguished element, constitute a Polish space with respect to the (pointed) {\em Gromov-Hausdorff metric} $d_{\scriptsize{\mbox{GH}}}$.

We shall briefly introduce the Gromov-Hausdorff metric and refer the readers to \cite{BBI:01,E:05} for a full description of this metric. Given two compact metric spaces $(X,d_2)$ and $(Y,d_2)$, a {\it correspondence} between $X$ and $Y$ is a subset $R\subset X\times Y$ such that for any $x\in X$, there is a $y\in Y$ with $(x,y)\in R$ and conversely for any $y\in Y$, there is an $x\in X$ with $(x,y)\in R$. The {\it distortion} of the correspondence is defined as follows:
\begin{align}\label{E:dis1}
\mbox{dis}(R)=\sup\{\vert d_1(x_1,x_2)-d_2(y_1,y_2)\vert:(x_1,y_1),(x_2,y_2)\in R\}.
\end{align}
Given two pointed compact metric spaces $(X,d_1,r_1)$ and $(Y,d_2,r_2)$, we define the
Gromov-Hausdorff distance between the pointed compact metric spaces $(X,d_1,r_1)$ and $(Y,d_2,r_2)$ by
\begin{align}\label{E:GH}
d_{\mbox{\scriptsize{GH}}}((X,d_1,r_1),(Y,d_2,r_2))=\frac{1}{2}\inf_{R}\mbox{dis}(R)
\end{align}
where $R$ ranges over all correspondences between $X$ and $Y$ such that $r_1$ and $r_2$ corresponds to each other. The Gromov-Hausdorff metric of two compact spaces $(X,d_1)$ and $(Y,d_2)$ is then defined to be (\ref{E:GH}) without $r_1,r_2$.

A pointed metric space $(X, d_1, r_1)$ may be rescaled by multiplying the metric with a positive constant $a$. We shall denote the rescaled space $(X, ad_1, r_1)$ in the following simply by $a X$.

Let $T$ be a Galton-Watson tree, we say $T$ is {\it critical} if the offspring distribution $\xi$ of $T$ satisfies $\mathbb{E}\xi=1$. In fact, $T$ is almost surely finite if and only if $\mathbb{E}\xi\le 1$. Let ${\sf supp}(\xi)=\{m\,\vert\,\mathbb{P}(\xi=m)>0\}$
denote the {\it support of} $\xi$ and define the {\it span}, denoted by ${\sf span}(\xi)$, as the greatest common divisor of $\{m\,\vert \,m\in{\sf supp}(\xi)\}$. If a Galton-Watson tree $T$ is finite, then
\begin{align}\label{E:span}
\vert T\vert=1+\sum_{v\in V(T)}d_{T}^{+}(v)\equiv 1\mod {\sf span}(\xi)
\end{align}
where $V(T)$ is the vertex set of $T$ and $d_{T}^{+}(v)$ represents the outdegree of $v$ in $T$. The convergence of a Galton-Watson tree $T_n$ conditioned on size $n$ (properly scaled) to $\CMcal{T}_{{\sf e}}$ is due to Aldous \cite{Ald:3}.
\begin{theorem}\label{T:tran}
Let $T_n$ be a Galton-Watson tree conditioned on having $n$ vertices, where $T_n$ is critical and the offspring distribution $\xi$ of $T_n$ has finite variance $\mathbb{V}\mbox{ar}\,\xi=\sigma^2$. As $n$ tends to infinity,  $T_n$ with edges rescaled to length $\sigma/(2\sqrt{n})$ converges in distribution to the CRT, i.e.,
\begin{eqnarray*}\label{E:crt}
\frac{\sigma}{2\sqrt{n}}T_n\xrightarrow{d} \CMcal{T}_{{\sf e}}
\quad\,\mbox{in the metric space }\,(\mathbb{K}^{\bullet},d_{\scriptsize{\mbox{GH}}}).
\end{eqnarray*}
\end{theorem}
The Galton-Watson tree conditioned on having $n$ vertices is also called the {\em conditioned Galton-Watson tree}. The conditioned Galton-Watson trees are essentially the same as the random {\em simply generated trees}; see \cite{De:98,Drmotabook}.
\subsection{Local convergence}
\label{ss:loco}
Let $\mathcal{X}$ denote the collection of rooted graphs that are connected and locally finite. Given two rooted graphs $G^{*}=(G,v_G)$ and $H^{*}=(H,v_H)$ from $\mathcal{X}$, we define
the distance \[{\sf d}(G^*, H^*)=2^{-\sup \{m\in \mathbb{N}_0 \,\mid\, U_m(G^*) \simeq U_m(H^*)\}}\]
where  $U_m(G^*)$ denotes the rooted subgraph of $G$ induced by all vertices with graph-distance at most $m$ from the root-vertex $v_G$, and $U_m(G^*) \simeq U_m(H^*)$ represents that the two subgraphs are isomorphic as rooted graphs.
The distance ${\sf d}$ satisfies the axioms of a premetric and two elements from $\mathcal{X}$ have distance zero from each other if and only if they are isomorphic as rooted graphs. Hence ${\sf d}$ defines a complete and separable metric on the collection of all isomorphism classes of graphs from $\mathcal{X}$ \cite{BBI:01,E:05}. %We shall define the local convergence in distribution in this metric space.

A random rooted graph ${\sf G}_n^{*}=({\sf G}_n,v_n)$ from $\mathcal{X}$ converges in the local weak sense toward a random element from ${\sf G}_{\infty}^{*}=({\sf G}_{\infty},v_{\infty})$, denoted by
$$({\sf G}_n,v_n)\xrightarrow{d} ({\sf G}_{\infty},v_{\infty}),$$
if the corresponding isomorphism classes converge weakly with respect to this metric. This is equivalent to requiring that for all fixed positive number $r$, and for all rooted graphs $(G,v)$ it holds that
\begin{align}\label{E:localdef}
\lim_{n\rightarrow\infty}\mathbb{P}[U_r({\sf G}_n,v_n)\simeq (G,v)]=\mathbb{P}[U_r({\sf G}_{\infty},v_{\infty})\simeq (G,v)].
\end{align}
%In the special case where the vertex $v_n$ is drawn uniformly at random from ${\sf G}_n$, we say $({\sf G}_\infty, v_\infty)$ is the Benjamini-Schramm limit of the sequence of random graphs $({\sf G}_n)_n$.

%%%%%%%%%%%%%%
\section{Proof of Theorem~\ref{T:crtk}}\label{S:proof}
We recall that ${\sf C}_{n,k}$ is a random $(k,\Omega)$-front coding tree of size $n$ that is uniformly selected from the class $\CMcal{C}_{n,k}$ and the size $L_{n,k}$ of the largest $\circ-\bullet$ $(k,\Omega)$-front coding tree in ${\sf C}_{n,k}$ satisfies (\ref{E:largest}). This implies that the Gromov-Hausdorff distance between ${\sf C}_{n,k}$ and ${\sf L}_{n,k}$ is bounded by $n^{\varepsilon}$ with high probability. If we choose $\varepsilon=1/4$, it follows that
\begin{align}\label{E:ghbc}
d_{\scriptsize{\mbox{GH}}}({\sf L}_{n,k}n^{-1/2},{\sf C}_{n,k}n^{-1/2})\xrightarrow{p} 0.
\end{align}
Let ${\sf B}_{n,k}$ denote a random $\circ-\bullet$ $(k,\Omega)$-coding tree that is uniformly chosen from all the $\circ-\bullet$ $(k,\Omega)$-coding trees of size $n$, so in order to establish the convergence of rescaled ${\sf C}_{n,k}$ to $\CMcal{T}_e$, from (\ref{E:ghbc}) it suffices to show that for the rescaled ${\sf B}_{n,k}$.

First we generate the random $\circ-\bullet$ $(k,\Omega)$-front coding tree by the Boltzmann sampler $\Gamma B_k(\rho_{k,\Omega})$. %The $\Omega$-$k$-tree ${\sf G}_{n,k}^{\square}$ is the corresponding random rooted $\Omega$-$k$-tree of ${\sf C}_{n,k}$ under the bijection $\varphi^{-1}:\CMcal{C}_{n,k}\rightarrow \CMcal{G}_{n,k}^{\square}$. Finally,
Let ${\sf T}_{n}$ be the {\it black tree} obtained from ${\sf B}_{n,k}$ by replacing every edge $\bullet-\circ-\bullet$ by an edge $\bullet-\bullet$ which keeps the labels on the black nodes, consequently black trees are in bijection with $\circ-\bullet$ $(k,\Omega)$-front coding trees; see Figure~\ref{F:1}.

From the construction of the Boltzmann sampler $\Gamma B_k(\rho_{k,\Omega})$, it is clear that any black node has $k$ white children and the number of black children $\xi_{\circ}$ of the white node in ${\sf B}_{n,k}$ follows the probability distribution (\ref{E:bullet}). This implies, the black grandchildren $\xi_{\bullet}$ of any black node has the probability distribution
\begin{align}\label{E:md1}
\mathbb{P}[\xi_{\bullet}=i]=\mathbb{P}[\sum_{j=1}^k\xi_{\circ,j}=i]\quad\mbox{ and }\quad\,
\xi_{\circ,j}\stackrel{d}{=}\xi_{\circ}.
\end{align}
Furthermore, (\ref{E:md1}) is exactly the offspring distribution of the black tree ${\sf T}_{n}$, thus from (\ref{E:domieq}) we know that $\mathbb{E}\,\xi_{\bullet}=k\mathbb{E}\,\xi_{\circ}=1$ and ${\sf T}_{n}$ is a critical Galton-Watson tree with span $\gcd(\Omega_{\scriptsize{\mbox{out}}})$ where $\gcd(\Omega_{\scriptsize{\mbox{out}}})$ denotes the greatest common divisor of the integers in $\Omega_{\scriptsize{\mbox{out}}}$.

We denote by ${\sf G}_{n,k}^{\bullet}$ the $\Omega$-$k$-tree that corresponds to the random $(k,\Omega)$-coding tree ${\sf B}_{n,k}$ under the bijection $\varphi^{-1}:{\sf B}_{n,k}\mapsto{\sf G}_{n,k}^{\bullet}$. For any two black nodes $x,y$ in ${\sf B}_{n,k}$, we set $d_{{\sf B}_{n,k}}(x,y)={\rm dist}_{{\sf T}_{n}}(x,y)$, where ${\rm dist}$ denotes the usual graph theoretical distance. For the case $k\ne 1$, the distance $d_{{\sf B}_{n,k}}(x,y)$ of two black nodes $x,y$ in ${\sf B}_{n,k}$ is different from the distance ${\rm dist}_{{\sf G}_{n,k}^{\bullet}}(x,y)$ of $x,y$ in the original $\Omega$-$k$-tree ${\sf G}_{n,k}^{\bullet}$. In order to represent the distances ${\rm dist}_{{\sf G}_{n,k}^{\bullet}}(x,y)$ for any two black nodes $x,y$ in the tree ${\sf B}_{n,k}$, we need to decompose ${\sf B}_{n,k}$ into {\it  blocks} according to the distance table from Algorithm $1$. We implement the Algorithm $1$ on the random tree ${\sf B}_{n,k}$ to have every black node marked with a distance and every white node marked with a distance sequence. For this random tree ${\sf B}_{n,k}$, denote by
${\sf B}_{i,n,k}$ a subtree of ${\sf B}_{n,k}$ that we call  {\it an $i$-block}\,:
\begin{enumerate}
\item ${\sf B}_{1,n,k}$ is rooted at the root and is induced by the root and all the black nodes that are in distance one to the vertex $1$.
\item ${\sf B}_{i,n,k}$, $i\ge 2$, is rooted at a white node with distance sequence $((i-1)^{k})$ and is induced by this node and all its black descendants that have distance $i$ to the vertex $1$.
\end{enumerate}
By construction, there is only one subtree ${\sf B}_{1,n,k}$ in ${\sf B}_{n,k}$, but there could be many subtrees ${\sf B}_{i,n,k}$ of ${\sf B}_{n,k}$ for $i\ne 1$; see Figure~\ref{F:3}.
%We will call each subtree ${\sf C}_{i,n,k}$ a {\em block} of ${\sf C}_{n,k}$.
For any two black nodes $x,y$ in ${\sf B}_{n,k}$, let $\delta_{{\sf B}_{n,k}}(x,y)=a-1$ where $a$ is the minimal number of blocks necessary to cover the path connecting $x$ and $y$. In particular if $x,y$ are in the same block of ${\sf B}_{n,k}$, then $\delta_{{\sf B}_{n,k}}(x,y)=0$. The following lemma will show that, for any two black nodes $x,y$, the distance ${\rm dist}_{{\sf G}_{n,k}^{\bullet}}(x,y)$ is almost the same as the block-distance $\delta_{{\sf B}_{n,k}}(x,y)$.
\begin{figure}[htbp]
\begin{center}
\includegraphics[scale=0.7]{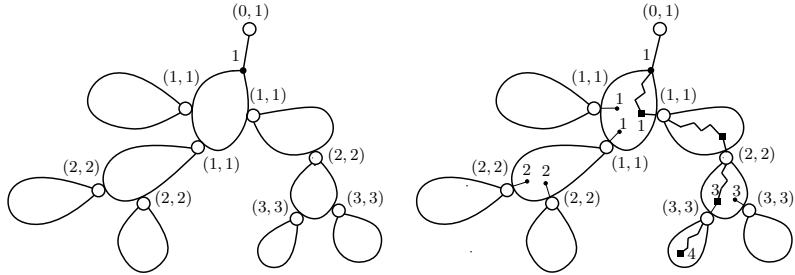}
\caption{ A decomposition of a random $(2,\Omega)$-front coding tree ${\sf B}_{n,2}$ into blocks ${\sf B}_{i,n,2}$ (left) where the pair $(a,b)$ of integers represents the distance sequence on the root of a block. A spine (right) consists of selected good nodes in ${\sf B}_{n,2}$.
\label{F:3}}
\end{center}
\end{figure}
\begin{lemma}\label{L:diseq}
Let ${\sf B}_{n,k}$ denote the tree corresponding to the Boltzmann sampler $\Gamma B_k(\rho_{k,\Omega})$ conditioned on having $n$ black nodes, let ${\sf G}_{n,k}^{\bullet}$ be the corresponding $\Omega$-$k$-tree of ${\sf B}_{n,k}$ under the bijection $\varphi^{-1}:\mathsf{B}_{n,k}\mapsto \mathsf{G}_{n,k}^{\bullet}$. Then for any two black nodes $x,y$ in ${\sf G}_{n,k}^{\bullet}$,
\begin{equation}\label{E:diseq}
{\rm dist}_{{\sf G}_{n,k}^{\bullet}}(x,y)=\delta_{{\sf B}_{n,k}}(x,y)+i\,\,\mbox{ where }\, i\in \{0,1,2,3\}.
\end{equation}
\end{lemma}
\begin{proof}
If $x,y$ are in the same block, i.e., $\delta_{{\sf B}_{n,k}}(x,y)=0$. If both of them are in a block ${\sf B}_{1,n,k}$, then
$${\rm dist}_{{\sf G}_{n,k}^{\bullet}}(x,y)\le {\rm dist}_{{\sf G}_{n,k}^{\bullet}}(x,1)+{\rm dist}_{{\sf G}_{n,k}^{\bullet}}(y,1)=2=\delta_{{\sf B}_{n,k}}(x,y)+2.$$
If both of them are in a block ${\sf B}_{i+1,n,k}$ for some $i\ge 1$, recall that the root of ${\sf B}_{i+1,n,k}$ is a white node with distance sequence $(i^k)$. Suppose the root of ${\sf B}_{i+1,n,k}$ has label $\{a_1,a_2,\ldots,a_k\}$, then for $x\in {\sf B}_{i+1,n,k}$, there exists an integer $p$ such that ${\rm dist}_{{\sf G}_{n,k}^{\bullet}}(a_p,x)=1$. Otherwise if for all $m\le k$, ${\rm dist}_{{\sf G}_{n,k}^{\bullet}}(a_m,x)>1$. It follows that ${\rm dist}_{\varphi({\sf B}_{n,k})}(x,1)>i+1$, which contradicts to the fact $x\in {\sf B}_{i+1,n,k}$. Similarly, there is an integer $q$ such that ${\rm dist}_{{\sf B}_{n,k}^{\bullet}}(a_q,y)=1$. Consequently
$${\rm dist}_{{\sf G}_{n,k}^{\bullet}}(x,y)\le
{\rm dist}_{{\sf G}_{n,k}^{\bullet}}(a_p,x)+{\rm dist}_{{\sf G}_{n,k}^{\bullet}}(a_q,y)+{\rm dist}_{{\sf G}_{n,k}^{\bullet}}(a_q,a_p)=3,$$
which implies (\ref{E:diseq}).

If $x,y$ are not in the same block, let $b$ be the last common parent of $x$ and $y$ in ${\sf B}_{n,k}$, then $b$ must be a black node. Let $a_1$ (resp. $b_1$) be the second black node on the path $b-\circ-a_1-\cdots-\circ-x$ (resp. $b-\circ-b_1-\cdots-\circ-y$) in ${\sf B}_{n,k}$. Then
one of the minimal paths connecting $x$ and $y$ in ${\sf G}_{n,k}^{\bullet}$ must pass node $b$. This is true because the $\Omega$-$k$-tree corresponding to the subtree of ${\sf B}_{n,k}$ rooted at $a_1$ and the $\Omega$-$k$-tree corresponding to the subtree of ${\sf B}_{n,k}$ rooted at $b_1$ are completely disjoint in ${\sf G}_{n,k}^{\bullet}$. This implies
$${\rm dist}_{{\sf G}_{n,k}^{\bullet}}(x,y)={\rm dist}_{{\sf G}_{n,k}^{\bullet}}(x,b)+{\rm dist}_{{\sf G}_{n,k}^{\bullet}}(y,b).$$
Suppose $x\in {\sf B}_{i+1,n,k}$, there must exist a black node $v_1$ on the path $b-\circ-a_1-\cdots-\circ-x$, such that ${\rm dist}_{{\sf G}_{n,k}^{\bullet}}(x,v_1)=1$ and $v_1\in {\sf B}_{i,n,k}$. For the node $v_1$, there exists a black node $v_2$ on the path such that $v_2\in {\sf B}_{i-1,n,k}$ and ${\rm dist}_{{\sf G}_{n,k}^{\bullet}}(x,v_2)=2$. We continue this process until we reach a black node $v_t$ such that $v_t$ and $b$ are in the same block. Similarly, we can find a sequence of black nodes $w_1,\ldots,w_s$ from different blocks such that $w_s$ and $b$ are in the same block and ${\rm dist}_{{\sf G}_{n,k}^{\bullet}}(y,w_s)=s$. It follows that
\begin{align*}
{\rm dist}_{{\sf G}_{n,k}^{\bullet}}(x,b)+{\rm dist}_{{\sf G}_{n,k}^{\bullet}}(y,b)
&=\delta_{{\sf B}_{n,k}}(x,b)+\delta_{{\sf B}_{n,k}}(y,b)+{\rm dist}_{{\sf G}_{n,k}^{\bullet}}(v_t,w_s)\\
&=\delta_{{\sf B}_{n,k}}(x,y)+{\rm dist}_{{\sf G}_{n,k}^{\bullet}}(v_t,w_s).
\end{align*}
Since $v_t$ and $w_s$ are in the same block, we have ${\rm dist}_{{\sf G}_{n,k}^{\bullet}}(v_t,w_s)\le 3$ and the proof is complete.
\end{proof}
Lemma~\ref{L:diseq} allows us to transfer the distance ${\rm dist}_{{\sf G}_{n,k}^{\bullet}}(x,y)$ of two vertices $x,y$ in a random $\Omega$-$k$-tree ${\sf G}_{n,k}^{\bullet}$ to the distance $\delta_{{\sf B}_{n,k}}(x,y)$ of two blocks in a random tree ${\sf B}_{n,k}$. In order to prove the convergence of ${\sf G}_{n,k}^{\bullet}$ to the CRT $\CMcal{T}_{{\sf e}}$, it is sufficient to prove that with high probability the difference between $\mathfrak{m}_k\delta_{{\sf B}_{n,k}}(x,y)$ and  ${\rm dist}_{{\sf T}_{n}}(x,y)$ is uniformly small for all choices of $x,y$, where ${\sf T}_{n}$ is the above conditioned critical Galton-Watson tree and $\mathfrak{m}_k$ is a constant. For this purpose we consider the {\em spine} of a size-biased enriched tree, which was adapted from the size-biased Galton-Watson tree. This idea has been used in studying the scaling limit of random graphs from subcritical graph classes \cite{P:14} and was further generalized to the random $\CMcal{R}$-enriched trees \cite{Stu:142}.

In fact, the block-distance $\delta_{{\sf B}_{n,k}}(v,1)$ to the vertex $1$ in the random tree ${\sf B}_{n,k}$ is not related to the depth of $v$ in ${\sf B}_{n,k}$. It turns out that we have to choose a {\em good} black node $\nu_i$ from a block ${\sf B}_{i,n,k}$ of the random tree ${\sf B}_{n,k}$, such that they form a {\em spine} $\nu_1,\ldots,\nu_m$ and $\delta_{{\sf B}_{n,k}}(\nu_i,1)$ increases as the depth of $\nu_i$ on this spine increases; see Fig~\ref{F:3} and \ref{F:4}.

We call a black node $v$ in a $(k,\Omega)$-front coding tree {\em good} if one of its white children has distance sequence $(i^k)$ for some integer $i\ge 1$. Let ${\sf B}_k$ denote the random $(k,\Omega)$-front coding tree that is generated by the above Boltzmann sampler
so that ${\sf B}_{n,k}=({\sf B}_k:\vert{\sf B}_k\vert=n)$. In the same way, let ${\sf C}_{i,k}$ be a block of ${\sf B}_k$ which equals ${\sf B}_{i,n,k}$ if we condition ${\sf B}_k$ on size $n$. The next Lemma~\ref{L:dis1} will enable us to construct a size-biased enriched tree.
\begin{lemma}\label{L:dis1}
Suppose that $i\ge 1$ and let $\xi_{k,i}$ be the random variable counting the number of good black nodes $v$ in an $i$-block ${\sf C}_{i,k}$ in ${\sf B}_k$. Then
$\mathbb{E}\,\xi_{k,i}=1$.
\end{lemma}
\begin{proof}
The offspring $\xi_{\circ}$ of every white root in ${\sf B}_k$ follows probability distribution (\ref{E:bullet}) and the offspring of every black node in ${\sf B}_k$ is distributed as the sum of $k$ independent and identically distributed random variables $\xi_{\circ,i}$ which are copies of $\xi_{\circ}$. The distance sequence on every white node of ${\sf B}_k$ determines if its children (black nodes) are good or not. We first compute $\mathbb{E}\,\xi_{k,1}$. Together with (\ref{E:domieq}), the first generation of the white root of ${\sf C}_{1,k}$ has
\begin{align}
\label{eq:fm}
\mathbb{E}(\xi_{\circ})
&=\sum_{i\in \Omega_{\scriptsize{\mbox{out}}}}i\,(C_k^{\circ}(\rho_{k,\Omega}))^{-1}
\,\frac{(B_k(\rho_{k,\Omega}))^i}{i!}=k^{-1}
\end{align}
expected number of black nodes. We assume that $\mu_1$ is a black node in the first generation, $\mu_1$ has $k$ white-node children in ${\sf C}_{1,k}$, among which $(k-1)$ white nodes have distance sequence $(0^{k-2},1^{2})$ and they have $\mathbb{E}\,\xi_{k-1,1}$ expected number of good black descendants in ${\sf B}_{k}$. One white-node child has distance sequence $(0^{k-1},1)$ and it has $\mathbb{E}\,\xi_{k,1}$ expected number of good black descendants in ${\sf B}_{k}$. It follows that
$\mathbb{E}\,\xi_{k,1}=k^{-1}\mathbb{E}\,\xi_{k,1}+(1-k^{-1})\mathbb{E}\,\xi_{k-1,1}$
which implies $\mathbb{E}\,\xi_{k,1}=\mathbb{E}\,\xi_{2,1}$. It is easy to compute $\mathbb{E}\,\xi_{2,1}$ by repeating the same procedure, which yields
$\mathbb{E}\,\xi_{k,1}=\mathbb{E}\,\xi_{2,1}=1$. Similarly, we can show for $i\ne 1$, $\mathbb{E}(\xi_{k,i})=k^{-1}\cdot k\cdot \mathbb{E}(\xi_{k,1})=1$.
\end{proof}
\begin{figure}[htbp]
\begin{center}
\includegraphics[scale=0.7]{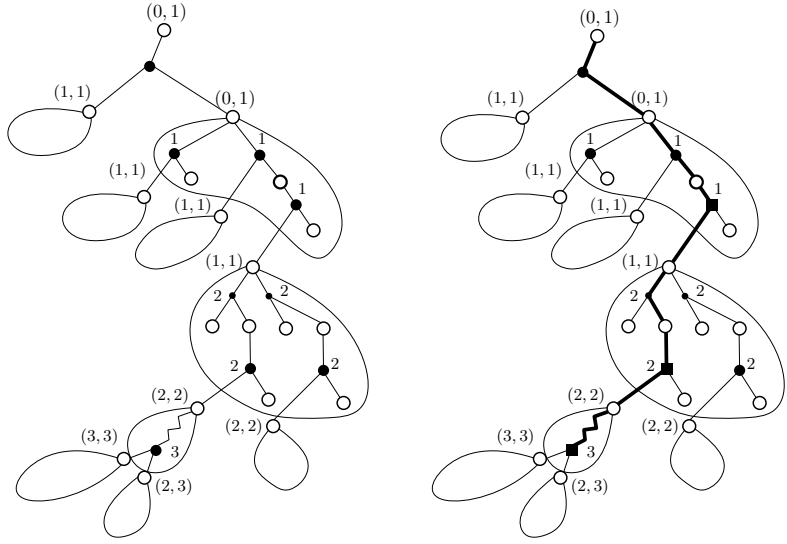}
\caption{ A $(2,\Omega)$-front coding tree ${\sf B}_{2}$ with good nodes drawn with black dots (left) and a size-biased enriched tree $\hat{{\sf B}}^{(3)}_2$ with
a spine consisting of selected good nodes drawn with black squares (right).
\label{F:4}}
\end{center}
\end{figure}
We will next define a {\em size-biased enriched tree} $\hat{{\sf B}}^{(m)}_k$ from a random $(k,\Omega)$-front coding tree ${\sf B}_k$. This construction is adapted from \cite{J:12}, which is a truncated version of the infinite size-biased Galton-Watson tree introduced by Kesten \cite{Kesten}, Lyons, Pemantle and Peres \cite{Lyons}. The size-biased Galton-Watson tree considered the distribution of offsprings in each generation of a Galton-Watson tree, while our size-biased enriched tree considered the distribution of good nodes in every block of ${\sf B}_k$. Let $\hat{\xi}_{k,i}$ be a random variable with the {\em size-biased} distribution
\begin{equation}\label{E:size}
\mathbb{P}(\hat{\xi}_{k,i}=q)=q\,\mathbb{P}(\xi_{k,i}=q).
\end{equation}
The expected value $\mathbb{E}\,\xi_{k,i}=1$ in Lemma~{\ref{L:dis1}} guarantees that $\hat{\xi}_{k,i}$ is a probability distribution on the set $\mathbb{N}_0=\{0,1,2,\ldots\}$.

The size-biased enriched tree $\hat{{\sf B}}^{(m)}_k$ is now defined as follows. It starts with a {\it mutant block} ${\sf C}_{1,k}$ which is rooted at a usual root (that has distance sequence $(0,1^{k-1})$) and contains good nodes.
We now choose one of these good nodes (which number is distributed according to $\hat{\xi}_{k,1}$) and call it {\it heir}
(and also {\it mutant}).
The block ${\sf C}_{2,k}$ that is rooted at the child with distance sequence $(1^k)$ of this heir will be the next mutant block, where we again assume that it has at least one good node. All other blocks that are adjacent to
${\sf C}_{1,k}$ are {\it normal}. We again choose one of the good nodes of the mutant block  ${\sf C}_{2,k}$
(which number is distributed according to $\hat{\xi}_{k,2}$) and proceed inductively to
choose mutant blocks and heirs till ${\sf C}_{m,k}$. All other blocks stay normal.
We denote the heir in the $m$-th mutant block ${\sf C}_{m,k}$ by $h$.
The path from the root to $h$ is  called {\em spine} of $\hat{{\sf B}}^{(m)}_k$; see Figure~\ref{F:4}.

The probability that a given mutant block contains $q$ good nodes and one of them is chosen as heir is, see (\ref{E:size}), $q^{-1}\mathbb{P}(\hat{\xi}_{k,i}=q)=\mathbb{P}(\xi_{k,i}=q)$.
For any given random $(k,\Omega)$-front coding tree $T$, let $T^{\alpha}$ denote the tree $T$ with a fixed spine $\alpha$ of block-depth $m$. Then the probability
\begin{equation}\label{E:spine}
\mathbb{P}(\hat{{\sf B}}^{(m)}_k=T^{\alpha},\mbox{ with } \alpha\,\mbox{ as a spine})=\mathbb{P}({\sf B}_k=T).
\end{equation}
This shows, once the spine is fixed, that the probability that the size biased tree $\hat{{\sf B}}^{(m)}_k$ equals $T^{\alpha}$ is the same as the probability of generating $T$. In fact, (\ref{E:spine}) is true for any fixed spine $\alpha$; see Eq.(3.2) in \cite{J:12}. We will need (\ref{E:spine}) to build a connection between $\mathfrak{m}_k\delta_{{\sf B}_{n,k}}(x,y)$ and $d_{{\sf B}_{n,k}}(x,y)$ with high probability in Lemma~\ref{L:k1}.
\begin{lemma}\label{L:k1}
Let $\CMcal{B}_{n,k}$ be the class of $\circ-\bullet$ $(k,\Omega)$-front coding trees of size $n$ such that the white root has label $\{1,2,\ldots,k\}$ and ${\sf B}_{n,k}\in \CMcal{B}_{n,k}$ is uniformly selected at random. Let $\mathfrak{m}_k=kH_{k}$. Then for all $s>1$ and $0<\epsilon<1/2$ with $2\epsilon s>1$, we have for all black nodes $x,y$ in ${\sf B}_{n,k}$ such that $x$ is an ancestor of $y$, that one of these two properties
\begin{eqnarray}\label{E:key1}
&&\delta_{{\sf B}_{n,k}}(x,y)\ge\log^s(n)\,\mbox{ and }\,\vert d_{{\sf B}_{n,k}}(x,y)-\mathfrak{m}_k\delta_{{\sf B}_{n,k}}(x,y)\vert\le \delta_{{\sf B}_{n,k}}(x,y)^{1/2+\epsilon},\\
\label{E:key2}
&&\delta_{{\sf B}_{n,k}}(x,y)<\log^s(n)\,\mbox{ and }\,d_{{\sf B}_{n,k}}(x,y)\le \log^{s+2}(n)
\end{eqnarray}
holds with high probability.
\end{lemma}
\begin{proof}
Suppose the opposite of (\ref{E:key1}) is true, that is, there exist black nodes $x,y$ in ${\sf B}_k$ such that $x$ is an ancestor of $y$ and they satisfy
\begin{eqnarray}\label{E:co1}
&&\delta_{{\sf B}_{k}}(x,y)\ge \log^s(|{\sf B}_{k}|)\, \mbox{ and }\,\vert d_{{\sf B}_{k}}(x,y)-\mathfrak{m}_k\delta_{{\sf B}_{k}}(x,y)\vert> \delta_{{\sf B}_{k}}(x,y)^{1/2+\epsilon}.
\end{eqnarray}
We will denote by $ \CMcal{F}_1$ the set of triples $({\sf B}_{k},x,y)$ (with $x,y$ in ${\sf B}_{k}$) that satisfy (\ref{E:co1}). Thus we just have to show that  $\mathbb{P}(({\sf B}_{k},x,y)\in \CMcal{F}_1 \big| |{\sf B}_{k}| = n)=o(1)$ as $n$ tends to infinity.

Recall that ${\sf B}_{n,k}$ is a random $(k,\Omega)$-front coding trees generated by the Boltzmann sampler $\Gamma B_k(\rho_{k,\Omega})$
with $n$ black nodes. Thus, in combination of Lemma~\ref{L:Boltz} and the universal analytic solution of functional equations; see Theorem 2.19 in \cite{Drmotabook}, it holds that for a positive constant $\sigma_{\Omega}^2=k\mathbb{V}\mbox{ar}\,\xi_{\circ}$,
\begin{align}\label{E:drawc}
\mathbb{P}[{\sf B}_{n,k}]=\mathbb{P}[\vert\Gamma B_k(\rho_{k,\Omega})\vert=n]
=\frac{b_{k,\Omega}(n)\rho_{k,\Omega}^n}{n!\,B_k(\rho_{k,\Omega})}
\sim \frac{n^{-3/2}}{\sigma_{\Omega}\sqrt{2\pi}}\,\mbox{ as }\, n\rightarrow \infty.
\end{align}
We apply (\ref{E:spine}) on the random $(k,\Omega)$-front coding tree ${\sf B}_{n,k}$ with a spine that connects $x$ to $y$. The block-depth of this spine is at least $\log^s n$ by assumption (\ref{E:co1}), which leads to
\begin{align}
\nonumber\mathbb{P}[({\sf B}_{k},x,y)\in \CMcal{F}_1 \big| |{\sf B}_{k}| = n]&\le\mathbb{P}[{\sf B}_{n,k}]^{-1}\sum_{m=\log^s n}^{n-1}
\mathbb{P}[(\hat{{\sf B}}^{(m)}_k,x,y)\in \CMcal{F}_1\,\mbox{ and }\, \vert\hat{{\sf B}}^{(m)}_k\vert=n]\\
&\sim \sigma_{\Omega}\sqrt{2\pi} n^{3/2}\sum_{m=\log^s n}^{n-1}
\mathbb{P}[(\hat{{\sf B}}^{(m)}_k,x,y)\in \CMcal{F}_1\,\mbox{ and }\, \vert\hat{{\sf B}}^{(m)}_k\vert=n]\label{E:fi}
\end{align}
as $n\rightarrow \infty$. Here the length of the spine in $\hat{{\sf B}}^{(m)}_k$ is distributed as the sum of $m$ independent random variables $\zeta_{1,k},\zeta_{2,k},\ldots,\zeta_{m,k}$ where each $\zeta_{i,k}$ is distributed as the length of the path from the selected good node in some block ${\sf C}_{i,k}$ to the root of this block. We have for $k\ge 2$, the probability generating functions of random variables $\zeta_{i,k}$ are
\begin{eqnarray*}
\mathbb{E}\,z^{\zeta_{1,k}}=\prod_{i=1}^{k-1}\frac{iz}{k-iz}\,\mbox{ where }\,i\ne 1\,\mbox{ and }\, \mathbb{E}\,z^{\zeta_{i,k}}=z\cdot\mathbb{E}\,z^{\zeta_{1,k}}.
\end{eqnarray*}
(We just have to extend the proof idea of Lemma~\ref{L:dis1}.)
For the case $k=1$, every $\zeta_{i,1}$ is distributed with probability $\mathbb{P}[\zeta_{i,1}=1]=1$. As an immediate consequence, $\zeta_{i,k}$ has finite exponential moments for every $i,k$ and $\mathbb{E}[\zeta_{i,k}]=kH_{k}$, $\mathbb{E}[\zeta_{1,k}]=kH_{k-1}$ for $k\ge 2,i\ne 1$ and $H_{k}$ is the $k$-th Harmonic number. For the case $k=1$ we have $\mathbb{E}[\zeta_{i,1}]=1$ for every $i$. We set $\mathfrak{m}_k=kH_{k}$ for $k\ge 1$. Furthermore, the assumption in (\ref{E:co1}) implies
\begin{eqnarray}\label{E:bound}
\mathbb{P}[(\hat{{\sf B}}^{(m)}_k,x,y)\in \CMcal{F}_1\,\mbox{ and }\, \vert\hat{{\sf B}}^{(m)}_k\vert=n]\le \mathbb{P}[\vert \sum_{i=1}^m\zeta_{i,k}-m\cdot\mathfrak{m}_k\vert>m^{1/2+\epsilon}].
\end{eqnarray}
By applying the deviation inequality (see \cite{J:12,P:14,P:142}) on the random variables $\zeta_{1,k},\zeta_{2,k},\ldots,\zeta_{m,k}$, we get for some positive constant $\bar{c}_1$ and $m\in [\log^sn,n-1]$,
\begin{eqnarray*}
\mathbb{P}[\vert \sum_{i=1}^m\zeta_{i,k}-m\cdot\mathfrak{m}_k\vert>m^{1/2+\epsilon}]
\le 2\exp(-\bar{c}_1(\log n)^{2s\epsilon})=o(n^{-5/2}).
\end{eqnarray*}
Together with (\ref{E:fi}) and (\ref{E:bound}), we can conclude that
 $\mathbb{P}[({\sf B}_{k},x,y)\in \CMcal{F}_1 \big| |{\sf B}_{k}| = n]=o(1)$.

Now we turn to suppose the opposite of (\ref{E:key2}) is true, i.e.,
there exist black nodes $x,y$ in ${\sf B}_{k}$ such that $x$ is an ancestor of $y$. They satisfy
\begin{equation}\label{E:co2}
\delta_{{\sf B}_{k}}(x,y)<\log^s(\vert{\sf B}_k\vert)\,\mbox{ and }\, d_{\CMcal{B}_{k}}(x,y)>\log^{s+2}(\vert{\sf B}_k\vert).
\end{equation}
We use the notation $\CMcal{F}_2$ to represent the set of triples $({\sf B}_{k},x,y)$ (with $x,y$ in ${\sf B}_{k}$) that satisfy (\ref{E:co2}). Again from (\ref{E:fi}) and from the deviation inequality, we obtain for some positive constant $\bar{c}_2$,
\begin{align*}
\mathbb{P}[({\sf B}_{k},x,y)\in \CMcal{F}_2 \big||{\sf B}_{k} | = n]
&\le \sigma_{\Omega}\sqrt{2\pi}n^{3/2}
\sum_{m=1}^{\log^s n}\mathbb{P}[(\hat{{\sf B}}^{(m)}_k,x,y)\in \CMcal{F}_2\,\mbox{ and }\, \big|\hat{{\sf B}}^{(m)}_k\vert=n]\\
&\le \sigma_{\Omega}\sqrt{2\pi}n^{3/2}\sum_{m=1}^{\log^s n}\mathbb{P}[\sum_{i=1}^m\zeta_{i,k}>\log^{s+2} n]\\
&=O(n^{3/2})(\log^s n)\exp(-\bar{c}_2\log^{2s+4}(n))=o\,(1)
\end{align*}
and the proof is complete.
\end{proof}
Now we are ready to prove our first main result.

\medskip\noindent
{\em Proof of Theorem~\ref{T:crtk}}. It follows from Lemma~\ref{L:k1} that with high probability
\begin{eqnarray*}
\vert d_{{\sf B}_{n,k}}(x,y)-\mathfrak{m}_k \delta_{{\sf B}_{n,k}}(x,y)\vert\le \delta_{{\sf B}_{n,k}}(x,y)^{1/2+\epsilon}+\log^{s+2}(n)
\end{eqnarray*}
holds for any fixed $s$ and $\epsilon$ such that $0<\epsilon<\frac{1}{2}$ and $2\epsilon s>1$, and holds for all black nodes $x,y$ where $x$ is an ancestor of $y$ in the random $(k,\Omega)$-front coding tree ${\sf B}_{n,k}$. For any two black nodes $\mu,\nu$ in ${\sf B}_{n,k}$, let $\alpha$ be the last common ancestor of $\mu$ and $\nu$ ($\alpha$ could be a white node of ${\sf B}_{n,k}$), then
\begin{align}
\nonumber\vert d_{{\sf B}_{n,k}}(\mu,\nu)-\mathfrak{m}_k \delta_{{\sf B}_{n,k}}(\mu,\nu)\vert
&\le \delta_{{\sf B}_{n,k}}(\mu,\alpha)^{1/2+\epsilon}
+\delta_{{\sf B}_{n,k}}(\nu,\alpha)^{1/2+\epsilon}+2\log^{s+2}(n)\\
\label{E:dif1}&\le 2\mbox{H}({\sf B}_{n,k})^{1/2+\epsilon}+2\log^{s+2}(n),
\end{align}
where $\mbox{H}({\sf B}_{n,k})$ is the height of random tree ${\sf B}_{n,k}$. We recall that $d_{{\sf B}_{n,k}}(\mu,\nu)={\rm dist}_{{\sf T}_{n}}(\mu,\nu)$. It is clear that $\mbox{H}({\sf B}_{n,k})=\mbox{H}({\sf T}_{n})$ and consequently, (\ref{E:dif1}) rewrites to
\begin{align*}
\vert {\rm dist}_{{\sf T}_{n}}(\mu,\nu)-\mathfrak{m}_k \delta_{{\sf B}_{n,k}}(\mu,\nu)\vert &\le 2\mbox{H}({{\sf T}_{n}})^{1/2+\epsilon}+2\log^{s+2}(n).
\end{align*}
%For any two black nodes $\mu,\nu$ in ${\sf C}_{n,k}$, let $\mu^{\circ},\nu^{\circ}$ be one of the white children of $\mu,\nu$, respectively. Then
%The diameter $\mbox{D}({\sf W}_{kn+1})$ of a random tree ${\sf W}_{kn+1}$ is less than the height $\mbox{H}({\sf W}_{kn+1})$ of ${\sf W}_{kn+1}$ multiplied by $2$.
The tree ${\sf T}_{n}$ contains all black nodes of ${\sf B}_{n,k}$ and it is a critical conditioned Galton-Watson tree. By applying the tails for the height of ${\sf T}_{n}$; see Theorem 1.2 in \cite{J:12} and left-tail upper bounds for the height in \cite{J:12}, we obtain the Gromov-Hausdorff distance
\begin{align*}
d_{\scriptsize{\mbox{GH}}}(n^{-1/2}{\sf T}_{n},n^{-1/2}\mathfrak{m}_k{\sf B}_{n,k})&\le \frac{1}{2}\max_{\mu,\nu} \vert n^{-1/2}{\rm dist}_{{\sf T}_{n}}(\mu,\nu)-
n^{-1/2}\mathfrak{m}_k \delta_{{\sf B}_{n,k}}(\mu,\nu)\vert \\
&\le n^{-1/2}\mbox{H}({\sf T}_{n})^{1/2+\epsilon}+n^{-1/2}\log^{s+2}(n)\xrightarrow{p} 0.
\end{align*}
Namely, for any fixed $\varepsilon$, the probability of the event $d_{\scriptsize{\mbox{GH}}}(n^{-1/2}{\sf T}_{n},n^{-1/2}\mathfrak{m}_k{\sf B}_{n,k})\le \varepsilon$ converges to $1$ as $n$ tends to infinity. Since $\xi_{\circ}$ has probability distribution (\ref{E:bullet}), for any specific degree set $\Omega$, the variance of the offspring distribution in the first generation of the random tree ${\sf T}_{n}$ is $\sigma_{\Omega}^2=k\mathbb{V}\mbox{ar}\,\xi_{\circ}$. Then it follows from Theorem~\ref{T:tran} that
\begin{equation*}
\frac{\sigma_{\Omega}{\sf T}_{n}}{2\sqrt{n}}\xrightarrow{d}\CMcal{T}_{e}
\quad\,\mbox{in the metric space }\,(\mathbb{K}^{\bullet},d_{\scriptsize{\mbox{GH}}}).
\end{equation*}
Hence from the convergence of Gromov-Hausdorff distance and with
the help of Lemma~{\ref{L:diseq}}, we get
\begin{eqnarray*}
\frac{\mathfrak{m}_k\sigma_{\Omega}}{2\sqrt{n}}{\sf B}_{n,k} \xrightarrow{d}\CMcal{T}_{e}\,\quad\mbox{ and }\quad\,\frac{\mathfrak{m}_k\sigma_{\Omega}}{2\sqrt{n}}{\sf G}_{n,k}^{\bullet}\xrightarrow{d}\CMcal{T}_{e}
\end{eqnarray*}
where ${\sf G}_{n,k}^{\bullet}$ is the corresponding rooted $\Omega$-$k$-tree of ${\sf B}_{n,k}$ under the bijection $\varphi^{-1}:{\sf B}_{n,k}\mapsto {\sf G}_{n,k}^{\bullet}$. In the beginning of Section~\ref{S:proof} and in subsection~\ref{ss:red} we know that it suffices to prove Theorem~\ref{T:crtk} for the random $k$-tree ${\sf G}_{n,k}^{\bullet}$ that is uniformly selected from $\CMcal{G}_{n,k}^{\bullet}$. This indicates
\begin{eqnarray*}
\frac{\mathfrak{m}_k\sigma_{\Omega}}{2\sqrt{n}}{\sf G}_{n,k}^{\circ}
\xrightarrow{d}\CMcal{T}_{e}\quad \,\mbox{ and }\quad \frac{\mathfrak{m}_k\sigma_{\Omega}}{2\sqrt{n}}{\sf G}_{n,k}\xrightarrow{d}\CMcal{T}_{e}\quad\,\mbox{ where }\,\mathfrak{m}_k=kH_{k}.
\end{eqnarray*}
In particular, if $\Omega=\mathbb{N}_0$, then $\sigma_{\mathbb{N}_0}=k\cdot\mathbb{V}\mbox{ar}(\xi_{\circ})=1$ where $\xi_{\circ}$ is Poisson distributed with parameter $k^{-1}$. The proof of Theorem~\ref{T:crtk} is complete.
\qed
\section{Proof of Theorem~\ref{T:local}}\label{L:proof}
In this section, we are going to construct an infinite $\Omega$-$k$-tree ${\sf G}_{\infty, k}$ that is rooted at a front of distinguishable vertices. We then establish the convergence of ${\sf G}_{n,k}^\circ$ toward this random graph in the sense, that for each fixed integer $\ell \ge 0$ the front-rooted sub-$(k,\Omega)$-tree $U_\ell({\sf G}_{n,k}^\circ)$ that is induced by all vertices with distance at most $\ell$ from the marked front, converges in distribution to the corresponding sub-$(k,\Omega)$-tree $U_\ell({\sf G}_{\infty, k})$ of the limit object.

By the discussion in Subsection~\ref{ss:red}, the random $\Omega$-$k$-tree ${\sf G}_{n,k}^\circ$ is up to relabeling distributed like the $\Omega$-$k$-tree ${\sf G}_{n,k}^\square$ that is rooted at a fixed front with labels from $1$ to $k$. Hence we only need to study the neighborhoods of the root-front. If we distinguish any fixed vertex of the marked front in ${\sf G}_{n,k}^\square$, for example the vertex with label $1$, and also distinguish a fixed vertex of the marked front in ${\sf G}_{\infty, k}$, then our limit may be interpreted as a classical local weak limit of a sequence of vertex-rooted random graphs as discussed in Subsection \ref{ss:loco}. This may be justified by the following two arguments. First, as rooted graphs, all $k$ possible vertex-rootings of ${\sf G}_{n,k}^\square$ are identically distributed, and we shall see below that the same is true for the limit ${\sf G}_{\infty, k}$. Second, the $\ell$-neighborhood of a vertex of any front-rooted $\Omega$-$k$-tree is always a subgraph of the $\ell$-neighborhood of the marked front, and hence the weak convergence of the neighborhoods of the front implies the weak convergence of the neighborhoods of the vertices.

\begin{figure}[h]
	\begin{center}
		\includegraphics[scale=0.8]{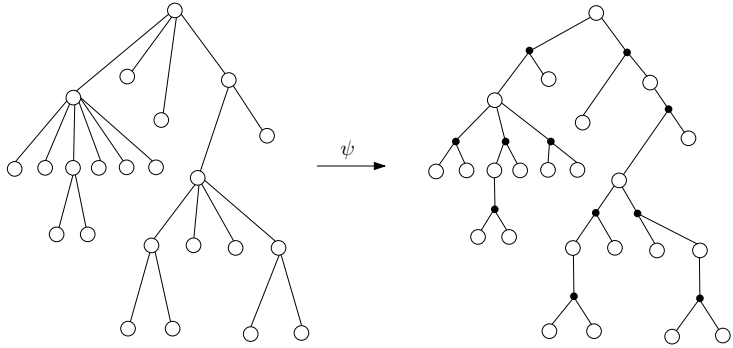}
		\caption{ The construction of $(k,\Omega)$-front coding trees out of plane trees where the outdegree of each vertex is a multiple of $k$, illustrated for the special case $k=2$.
			\label{F:new}}
	\end{center}
\end{figure}

The strategy of the proof is as follows. We may generate the random $\Omega$-$k$-tree ${\sf G}_{n,k}^\square$ by applying the bijection $\varphi^{-1}: \CMcal{C}_{n,k} \rightarrow  \CMcal{G}_{n,k}^{\square} $ to the  random $(k,\Omega)$-front coding tree $\mathsf{C}_{n,k}$. This random coding tree may be generated by conditioning a Boltzmann sampler $\Gamma C_k(\rho_{k,\Omega})$ on producing a coding tree with $n$ black vertices. We  observe that any ordered tree of white vertices where the outdegree of any vertex is a multiple of $k$ may be interpreted as a $(k,\Omega)$-front coding tree by adding black vertices in a canonical way. Here different plane trees may correspond to the same unlabelled $(k,\Omega)$-tree, but this will not be an issue. We may use this construction in order to formulate a coupling of the Boltzmann sampler $\Gamma C_k(\rho_{k,\Omega})$ with a Galton--Watson tree ${\sf T}_\circ$ that has a modified root-degree. If we condition this locally modified Galton--Watson tree on having $(kn + 1)$ vertices, then the result ${\sf T}_{n, \circ}$ corresponds, up to relabeling, to the $(k,\Omega)$-front coding tree $\mathsf{C}_{n,k}$. By the similar arguments as for the classical local convergence of simply generated trees, the random tree ${\sf T}_{n, \circ}$ converges weakly toward an infinite plane tree ${\sf T}_{\infty, \circ}$ that may be interpreted as a $(k,\Omega)$-coding tree ${\sf C}_{\infty,k}$ and consequently also as a front-rooted $\Omega$-$k$-tree ${\sf G}_{\infty, k}$. The final step in the proof is to deduce local convergence of the random $\Omega$-$k$-tree ${\sf G}_{n,k}^\square$ from this convergence of random trees.

The construction of a $(k,\Omega)$-front coding tree $\psi(T)$ out of a plane trees $T$, where the outdegree of each vertex is a multiple of $k$, is straight-forward. We canonically partition the offspring set of each vertex $v$ of $T$ into an ordered list of groups $G_1(v), G_2(v), \ldots$ of $k$ consecutive vertices. The edges between $v$ and its offspring are then deleted, and for each group $G_i(v)$ we add a black offspring vertex $u_i(v)$ to $v$ and add further edges such that $G_i(v)$ is the offspring set of $u_i(v)$. This construction is illustrated in Figure~\ref{F:new}.

We may now use this to formulate a coupling of Boltzmann distributed $(k,\Omega)$-front coding trees with a modified Galton--Watson tree. Similar as in Remark~\ref{re:boaut}, a Boltzmann sampler $\Gamma C_k(\rho_{k,\Omega})$ is given by starting with a white root, and connecting it with the roots of a random number $\eta_\circ$ of independent $\circ-\bullet$ $(k,\Omega)$-front coding trees where each is sampled according to an independent call to the Boltzmann sampler $\Gamma B_k(\rho_{k, \Omega})$ from Lemma~\ref{L:Boltz}. The distribution of $\eta_\circ$ is given by
\[
\mathbb{P}(\eta_\circ=i) = \frac{1}{C_k(\rho_{k, \Omega})} \frac{(B_k(\rho_{k,\Omega}))^i}{i!}
\]
for all $i \in \Omega$. Recall that the sampler in Lemma~\ref{L:Boltz} starts with a black node with $k$ white nodes as offspring. Each of the white nodes receives black offspring according to an independent copy of the random number $\xi_\circ$, whose distribution is given in \eqref{E:bullet}. Then the sampler recurs, that is, any black node in the youngest generation receives $k$ white vertices as offspring, each of which receives a random number (possibly zero) of black offspring, and so on.

Let ${\sf T}_\circ$ denote a modified Galton--Watson tree, where each vertex receives offspring according to an independent copy of $\xi := k \xi_\circ$, except for the root, which receives offspring according to $\eta := k \eta_\circ$. The order in which the recursion takes place in $\Gamma B_k(\rho_{k, \Omega})$ and $\Gamma C_k(\rho_{k,\Omega})$ does not matter, hence the $(k,\Omega)$-coding tree $\psi({\sf T}_\circ)$ is up to relabeling distributed like the $(k,\Omega)$-coding tree $\Gamma C_k(\rho_{k,\Omega})$. Moreover, if we let ${\sf T}_{n, \circ}$ denote the tree  ${\sf T}_\circ$ conditioned on having $(kn +1)$ vertices, then $\psi({\sf T}_{n,\circ})$ is distributed like the random $(k, \Omega)$-front coding tree ${\sf C}_{n,k}$.

Note that \eqref{eq:fm} implies that
$
	\mathbb{E}[\xi] =1,
$
and both $\xi$ and $\eta$ have finite exponential moments. We define the size-biased versions of these offspring distributions by
\[
	\mathbb{P}({\hat{\xi} = i}) = i \mathbb{P}({\xi=i}) \quad \text{and} \quad \mathbb{P}({\hat{\eta} = i}) = i \mathbb{P}(\eta=i) / \mathbb{E}[\eta].
\]
Let $T_{\infty, \circ}$ denote the following random infinite (but locally finite) plane tree. There are two types of non-root vertices, mutant and normal. The root receives offspring according to $\hat{\eta}$, and one of its sons is selected uniformly at random and declared mutant, whereas the others are normal. Normal vertices receive offspring according to an independent copy of $\xi$, all of which are normal. Mutant vertices receive offspring according to an independent copy of $\hat{\xi}$, among which one is selected uniformly at random and declared mutant, whereas the others are normal. Hence ${\sf T}_{\infty, \circ}$ is an infinite plane tree with a distinguished path that starts at the root and traverses the mutant vertices. We call this path the spine of ${\sf T}_{\infty, \circ}$.

We describe the convergence of the random tree ${\sf T}_{n, \circ}$ toward the limit tree ${\sf T}_{\infty, \circ}$ using a slight modification of the arguments in Janson's survey \cite{Janson:11}. For each plane tree $T$ and each integer $h \ge 0$ let $T^{[h]}$ denote the tree obtained by cutting away all vertices with height larger than $h$.
\begin{lemma}
	\label{le:trconv}
	For any integer $h \ge 0$, it holds that
	$
	{\sf T}_{n, \circ}^{[h]} \xrightarrow{d} {\sf T}_{\infty, \circ}^{[h]}.
	$
\end{lemma}
\begin{proof}
	It suffices to show for each plane tree $T$ with height  $h$ that
	\begin{align}
		\label{eq:tos}
		\lim_{n \to \infty} \mathbb{P}({{\sf T}_{n, \circ}^{[h]} = T}) = \mathbb{P}({{\sf T}_{\infty, \circ}^{[h]} = T}).
	\end{align}
	As ${\sf T}_{\infty, \circ}$ has infinite height, this already implies that $\mbox{H}({\sf T}_{n, \circ}) \ge h$ occurs with probability tending to $1$, and consequently 	$
	{\sf T}_{n, \circ}^{[h]} \xrightarrow{d} {\sf T}_{\infty, \circ}^{[h]}
	$. In order to check \eqref{eq:tos}, let $d_1, \ldots, d_t$ denote the depth-first-search ordered list of the degrees of all vertices in the pruned tree $T^{[h-1]}$. Moreover, let $(\xi_i)_{i \in \mathbb{N}}$ denote a family of independent copies of $\xi$. Set $N = kn +1$ and $D = d_1 + \cdots +d_t$. The probability $\mathbb{P}(|{\sf T}_{\circ}| =N, {\sf T}_{\circ}^{[h]} = T)$ is given by
	\begin{align}
	 \label{eq:tmp}
	 \mathbb{P}({\eta = d_1}) \left ( \prod_{j=2}^t \mathbb{P}({\xi = d_j}) \right ) \mathbb{P}(D + \sum_{j=t+1}^N \xi_j = N-1, D + \sum_{j=t+1}^m \xi_j \ge m \text{ for all $t <m<N$}).
	\end{align}
	A classical combinatorial observation, also called the cycle lemma, states that for any sequence $x_1, \ldots, x_{s} \ge -1$  of integers satisfying
		$
		\sum_{i=1}^s x_i = -r
		$
		for some $r \ge 1$, there are precisely $r$ integers $1 \le u \le s$ such that the cyclically shifted sequence $x_i^{(u)} = x_{1 + (i+u)\mod s}$ satisfies
		$
		\sum_{i=1}^\ell x_i^{(u)} > r
		$
		for all $1 \le \ell \le s-1$;
	see for example {\cite[Lem. 15.3]{Janson:11}}. Consequently, \eqref{eq:tmp} may be simplified to
	\begin{align}
		\label{eq:tmp1}
		\frac{D-t+1}{N-t} \mathbb{P}({\eta = d_1}) \left ( \prod_{j=2}^t \mathbb{P}({\xi = d_j}) \right ) \mathbb{P}(D + \sum_{j=t+1}^N \xi_j = N-1).
	\end{align}
	The tree $T$ has precisely $(D-t+1)$ vertices with height $h$. Hence the event ${\sf T}_{\infty, \circ}^{[h]} = T$ corresponds to precisely $(D-t+1)$ possible outcomes for the first $(h+1)$ levels of ${\sf T}_{\infty, \circ}$, depending on the location for the unique spine vertex with height $h$. Each has the same probability given by
	\[
		\mathbb{E}[{\eta}]^{-1}\mathbb{P}({\eta = d_1}) \prod_{j=2}^t \mathbb{P}({\xi = d_j}).
	\]
	Thus, $\mathbb{P}({\sf T}_{\infty, \circ}^{[h]} = T)=(D-t+1)\mathbb{E}[{\eta}]^{-1}\mathbb{P}({\eta = d_1}) \prod_{j=2}^t \mathbb{P}({\xi = d_j})$ and \eqref{eq:tmp} becomes
	\[
	\mathbb{P}(|{\sf T}_{\circ}| =N, {\sf T}_{\circ}^{[h]} = T) = \mathbb{P}({\sf T}_{\infty, \circ}^{[h]} = T) \frac{\mathbb{E}[\eta]}{N-t} \mathbb{P}(D + \sum_{j=t+1}^N \xi_j = N-1) .
	\]
	The central local limit theorem for the sum of independent identically distributed random integers yields that
	\[
		\mathbb{P}(D + \sum_{j=t+1}^N \xi_j = N-1) = (1 + o(1)) \frac{ k \gcd{(\Omega_{\scriptsize{\mbox{out}}}})}{\sqrt{2 \pi N \mathbb{V}\mbox{ar}[\xi]} }
	\]
	and consequently
	\begin{align}
	\label{eq:almost}
	\mathbb{P}(|{\sf T}_{\circ}| =N, {\sf T}_{\circ}^{[h]} = T) = (1 + o(1)) \mathbb{P}({\sf T}_{\infty, \circ}^{[h]} = T) n^{-3/2} \frac{\mathbb{E}[\eta]  \gcd{(\Omega_{\scriptsize{\mbox{out}}}})}{\sqrt{2 \pi k \mathbb{V}\mbox{ar}[\xi]} }.
	\end{align}
	Let $d(o)$ denote the root-degree of ${\sf T}_\circ$. It holds, since $\zeta$ has finite exponential moments, that $\mathbb{P}({\eta \ge \log(n)^2})$ is exponentially small. Hence, using the cycle lemma and central local limit theorem in an identical fashion as above, it follows that
	\begin{align*}
		\mathbb{P}({|{\sf T}_\circ|} = N)
		&= o(n^{-3/2}) + \sum_{d=1}^{\log(n)^2}\mathbb{P}({\eta = d})\frac{d}{N-1} \mathbb{P}({d + \sum_{j=2}^N \xi_j = N-1}) \\
		&= (1 + o(1)) n^{-3/2} \mathbb{E}[{\eta}] \frac{ \gcd{(\Omega_{\scriptsize{\mbox{out}}}})}{\sqrt{2 \pi k  \mathbb{V}\mbox{ar}[\xi]} },
	\end{align*}
which, together with \eqref{eq:almost}, implies \eqref{eq:tos} and we are done.
\end{proof}

We are now finally in the position to complete the proof of our second main theorem.
\begin{proof}[Proof of Theorem~\ref{T:local}]
	Let $\ell$ be an integer and let $G$ be an arbitrary finite unlabelled $\Omega$-$k$-tree that is rooted at a front. We claim that there exist an integer $L\ge 0$, that depends on both $\ell$ and $G$, and a set $\mathcal{E}$ of finite plane trees, such that any plane tree $T$, that corresponds to a $(k,\Omega)$-front coding tree $\psi(T)$ and hence to a front-rooted  $\Omega$-$k$-tree $G(T) := \varphi^{-1}(\psi(T))$, has the property $U_\ell(G(T)) = G$ if and only if $T^{[L]} \in \mathcal{E}$.
	
	This is certainly sufficient for deducing Theorem~\ref{T:local}, as by Lemma~\ref{le:trconv} it then follows that
	\[
		\lim_{n \to \infty} \mathbb{P}({ {\sf T}_{n, \circ}^{[L]} \in \mathcal{E}}) = \mathbb{P}({ {\sf T}_{\infty, \circ}^{[L]} \in \mathcal{E}})
	\]
	and consequently
	\[
		\lim_{n \to \infty} \mathbb{P}(U_\ell({\sf G}_{n,k}^\circ) = G) = \mathbb{P}(U_\ell({\sf G}_{\infty,k}) = G)
	\]
	with ${\sf G}_{\infty,k}$ denoting the $\Omega$-$k$-tree corresponding to ${\sf T}_{\infty, \circ}$.
	
	The reason why there exist such an integer $L$ and the set $\mathcal{E}$ is rather subtle. To each plane tree $T$ we may associate a unique sequence of increasing subtrees $T_0, T_1, \ldots$ of $T$ that all contain the root-vertex of $T$ and have the property $G(T_i) = U_i(G(T))$ for all $i$. Of course, the tree $T_\ell$ may, in general, have arbitrarily large height. However, in order to satisfy $G(T_\ell)=G$, the tree $T_\ell$ may not have more vertices, than the number of fronts in $G$. In particular, the height of $T_\ell$ is bounded by the number of fronts of $G$. Hence there exists a finite integer $L$ such that for any plane tree $T$ we may decide whether $U_\ell(G(T))=G$ by only looking at $T^{[L]}$.
\end{proof}

\section*{Acknowledgement}
We would like to thank three anonymous reviewers from Analco 16 for their very helpful suggestions and comments on the earlier version of this manuscript.

\end{document}